\documentclass[12pt]{amsart}
\usepackage{amsmath}
\usepackage{hyperref}
\usepackage{amssymb}
\usepackage{latexsym}
\usepackage{amscd}
\usepackage{graphicx} %We can use any other package if necessary
\usepackage{amsthm}
\usepackage{mathrsfs}
\usepackage{xypic}
\usepackage{bm}
\usepackage{color}
\usepackage{ulem}

\newdimen\AAdi%
\newbox\AAbo%
\def\AAk#1#2{\s_etbox\AAbo=\hbox{#2}\AAdi=\wd\AAbo\kern#1\AAdi{}}%
\def\AAr#1#2#3{\s_etbox\AAbo=\hbox{#2}\AAdi=\ht\AAbo\raise#1\AAdi\hbox{#3}}%
\font\tenmsb=msbm10 at 12pt \font\sevenmsb=msbm7 at 8pt
\font\fivemsb=msbm5 at 6pt
\newfam\msbfam
\textfont\msbfam=\tenmsb \scriptfont\msbfam=\sevenmsb
\scriptscriptfont\msbfam=\fivemsb
\def\Bbb#1{{\tenmsb\fam\msbfam#1}}
\newtheorem{thm}{Theorem}[section]
\newtheorem{lem}[thm]{Lemma}
\newtheorem{cor}[thm]{Corollary}

\newtheorem{pro}[thm]{Proposition}
\newtheorem{defi}[thm]{Definition}

\newcommand{\Section}[2]{\setcounter{equation}{0}
\allowdisplaybreaks
\section[#1]{#2}}

\def\pr {\noindent {\it Proof.} }

\def\n{\nabla}

\def\f#1#2{\frac{#1}{#2}}

\def\mc#1{\mathcal{#1}}
\def\pr{\frac {\partial}{\partial r}}

\def\td{\tilde}

\def\a{\alpha}

\def\p#1{\partial #1}

\def\de{\delta}
\def\De{\Delta}

\def\ep{\varepsilon}

\def\g{\gamma}
\def\k{\kappa}
\def\la{\lambda}
\def\La{\Lambda}
\def\om{\omega}
\def\Om{\Omega}
\def\th{\theta}
\def\Th{\Theta}

\def\w{\wedge}

\def\Hess{\mbox{Hess}}
\def\R{\Bbb{R}}

\def\lan{\langle}
\def\ran{\rangle}
\def\ra{\rightarrow}

\def\mb{\mathbf}

\def\Hess{\text{Hess }}

\subjclass[2010]{58E20,~53A10,~53C42.}
\begin{document}
\title
[On Lawson-Osserman Constructions] {On Lawson-Osserman Constructions}

\author
[Xiaowei Xu, Ling Yang and Yongsheng Zhang]{Xiaowei Xu, Ling Yang and Yongsheng Zhang}

\address{School of Mathematical Sciences,  University of Science and Technology of China, Hefei, 230026, Anhui
province, China;}
\address{and Wu Wen-Tsun Key Laboratory of Mathematics, USTC, Chinese Academy of
Sciences, Hefei, 230026, Anhui province, China.}
\email{xwxu09@ustc.edu.cn}

\address{School of Mathematical Sciences, Fudan University,
Shanghai 200433, China.} \email{yanglingfd@fudan.edu.cn}

\address{School of Mathematics and Statistics, Northeast Normal University, Changchun 130024, Jilin province, China.}
\email{yongsheng.chang@gmail.com}

\thanks{ }

\begin{abstract}
Lawson-Osserman \cite{l-o} constructed three types of non-parametric minimal cones of high codimensions based on Hopf maps between
 spheres, which correspond to Lipschitz but non-$C^1$ solutions to the minimal surface equations,
 thereby making
sharp contrast to the regularity theorem
%(see \cite{de,m1})
 for minimal graphs of codimension 1. In this paper, we develop the constructions
in a more general scheme. Once a mapping $f$ between unit spheres is composited of a harmonic Riemannian
submersion and a homothetic (i.e., up to a constant factor, isometric) minimal immersion, certain twisted
graph of $f$ can yield a non-parametric minimal cone. Because the choices of the second component usually form a huge moduli space,
%(see \cite{c-w,u,to})
our constructions produce a constellation of uncountably many examples. For each such cone, there exists an entire minimal graph whose tangent cone at infinity
is just the given one. Moreover, new phenomena on the existence, non-uniqueness and non-minimizing of solutions to the related Dirichlet problem are discovered.
%due to the amusing spiral asymptotic behaviour of a particular autonomous system on $\R^2$.

\end{abstract}
\maketitle

\tableofcontents

\renewcommand{\proofname}{\it Proof.}

%\small \parskip0.1mm\tableofcontents \normalsize\parskip3mm

\Section{Introduction}{Introduction}

The research on minimal graphs in Euclidean spaces has a long and fertile history.
Among others, the \textbf{Dirichlet problem} (cf. \cite{j-s,b-d-m, de, m1,l-o}) is a central topic in this subject:

\textit{Let $\Omega\subset \R^{d_1}$ be a bounded and strictly convex domain with boundary of class $C^r$ for $r\geq 2$.
        It asks, for a given function $f:\partial \Omega \rightarrow \mathbb R^{d_2}$ of class $C^s$ with $0\leq s\leq r$,
         what kind of and how many functions $\in C^0(\overline{\Omega};\R^{d_2})\bigcap \text{Lip}(\Omega;\R^{d_2})$
          exist
        so that each such $F$ satisfies the minimal surface equations in the weak sense (or equivalently, the graph of $F$ is minimal in the sense
        of \cite{al}) and $F|_{\partial \Omega}=f$.}

When $d_2=1$, we have a fairly profound understanding.

\begin{itemize}
\item Given arbitrary boundary data of class $C^0$,
                    by the works of J. Douglas \cite{d},
                    T. Rad\'o \cite{r,r2},
                    Jenkins-Serrin \cite{j-s}
                    and Bombieri-de Giorgi-Maranda \cite{b-d-m},
                    there exists a unique Lipschitz solution to the Dirichlet problem.
\item Furthermore, due to the works of E. de Giorgi \cite{de} and J. Moser \cite{m1}, this solution turns out to be analytic.
\item Each solution gives an absolutely area-minimizing graph
          by virtue of the convexity of $\Omega\times \mathbb R$ and \S5.4.18 of
          \cite{fe}.
          As a consequence, it is stable.
%\item Any entire minimal graph of dimension $n$ and codimension 1 has to be affine linear, provided that $n\leq 7$ or the slope function is bounded (see
%\cite{si} by J. Simons and \cite{m2} by J. Moser).
\end{itemize}

Utterly unlike the above, the situation for $d_2\geq 2$ becomes much more complicated.
          Even when $\Om=\mathbb D^{d_1}$ (the unit Euclidean disk),
          H. B. Lawson and R. Osserman \cite{l-o}
          discovered astonishing phenomena that reveal essential differences.

\begin{itemize}
          \item For $d_1=d_2=2$, some real analytic boundary data can be constructed
                   so that
                   there exist at least three different analytic solutions to the Dirichlet problem.
                   Moreover,
                   one of them corresponds to an unstable minimal surface.
          \item For $d_1\geq 4$ and $d_1-1\geq d_2\geq 3$, the Dirichlet problem is generally not solvable.
                   In fact, for each $f:S^{d_1-1}\ra S^{d_2-1}$ that is not homotopic to zero,
                   there exists a positive constant $c$ depending only on $f$, such that
                   the problem is unsolvable for the boundary data $f_\varphi:=\varphi\cdot f$, where
                   $\varphi$ is a constant no less than $c$.
          \item For certain boundary data, there exists a Lipschitz solution to the Dirichlet problem which is not $C^1$.
\end{itemize}

As shown in \cite{l-o}, the nonexistence and irregularity of the Dirichlet problem are intimately related as follows.
Given $f$ that represents a non-trivial element of $\pi_{d_1-1}(S^{d_2-1})$, the Dirichlet problem for $f_\varphi$ is solvable when $\varphi$ is small (due to the implicit function theorem)
but unsolvable for large $\varphi$. This leads Lawson-Osserman to suspect there exists a critical value $\varphi_0$
which supports some sort of singular solution.
In particular, for the Hopf map $H^{2m-1,m}:S^{2m-1}\rightarrow S^m$ with $m=2,4$ or $8$,
\begin{equation}
M_m:=\{ (\cos\th_m\cdot x,\sin\th_m\cdot H^{2m-1,m}(x)):x\in S^{2m-1}\}\subset S^{3m}
\end{equation}
with
\begin{equation}
\th_m:=\arccos\sqrt{\f{4(m-1)}{3(2m-1)}}
\end{equation}
is the principal orbit of maximal volume under certain group action,
and hence presents
 a minimal sphere (cf. W.Y. Hsiang \cite{hs}).
Then the minimal cone $C_m$ over $M_m$ is the graph of %$F_{m}: \R^{2m}\ra \R^{m+1}$
\begin{equation}F_{m}(y)=\left\{\begin{array}{cc}
\tan\th_m\cdot |y|\cdot H^{2m-1,m}(\f{y}{|y|}) & y\neq 0,\\
0 & y=0.
\end{array}\right.
\end{equation}
Hence, its restriction over $\Bbb{D}^{2m}$ gives a Lipschitz solution to the Dirichlet problem for boundary data $\tan\th_m\cdot H^{2m-1,m}$.

To develop constructions akin to Lawson-Osserman's in a more general framework,
we introduce the following concepts.

\begin{defi}
For a smooth map $f:S^n\ra S^m$,
if there exists an acute angle $\th$, such that
\begin{equation}
M_{f,\th}:=\{(\cos\th\cdot x,\sin\th\cdot f(x)):x\in S^n\}%\subset S^{n+m+1}
\end{equation}
is a minimal submanifold of $S^{n+m+1}$, then we call $f$ a {\bf Lawson-Osserman map (LOM)},
$M_{f,\th}$ the associated {\bf Lawson-Osserman sphere (LOS)}, and the cone $C_{f,\th}$ over $M_{f,\th}$ the corresponding {\bf Lawson-Osserman cone (LOC)}.
\end{defi}

Similarly, for an LOM $f$, the associated $C_{f,\th}$ is the graph of
%$F_{f,\th}: \R^{n+1}\ra \R^{m+1}$
\begin{equation}\label{cone-graph}
F_{f,\th}(y)=\left\{\begin{array}{cc}
\tan\th \cdot |y|\cdot f(\f{y}{|y|}) & y\neq 0,\\
0 & y=0.
\end{array}\right.
\end{equation}
Thus the restriction over $\Bbb{D}^{n+1}$ provides a Lipschitz solution to the Dirichlet problem
for the boundary data $ f_{\varphi_0}:=\varphi_0\cdot f$ with $\varphi_0:=\tan\th$.

 Assume $f:S^n\ra S^m$ is an LOM that is not a totally geodesic isometric embedding. Then $f$ is called an {\bf LOMSE} if the nonzero singular values
of $(f_*)_x$ are equal for each $x\in S^n$. As $x$ varies, these values give a continuous function $\la(x)$.
One can deduce that $\la(x)$ equals
 a constant $\la$ and that $f$ has constant rank $p$ (see Theorem \ref{g2} (ii)). Moreover, all components of this vector-valued function $f$, i.e. $f_1,\cdots,f_{m+1}$ are harmonic spherical functions of degree $k\geq 2$ (see Theorem \ref{npk}). Accordingly, we call such $f$ an {\bf LOMSE of (n,p,k)-type}. It is worth noting that, the Hopf map from $S^{2m-1}$ onto $S^m$ is an
LOMSE of $(2m-1,2m,2)$-type, for $m=2, 4 , 8$. Hence the LOMSEs and corresponding LOSs, LOCs are natural
generalizations of Lawson-Osserman's original constructions.

In this paper, we shall study LOMSEs systematically from several viewpoints.

A characterization of LOMSEs will be established in Theorem \ref{g2}, which asserts that
each of them can be written as the composition of a Riemannian submersion from $S^n$ with connected fibers and a homothetic minimal immersion into $S^m$.
In fact, the submersion, which determines $(n,p)$, has to be a Hopf fibration over a complex projective space, a quaterninonic projective space
or the octonionic projective line, according to the wonderful result in \cite{w};
 while the choices of the second component
 for each even integer $k$ usually form
              %a huge moduli space
a moduli space of large dimension (see \cite{c-w,oh,u,to,to2}), yielding a huge number of LOMSEs as well as the associated LOSs and LOCs. Note that except for the three original
Lawson-Osserman cones,
we always have $m>n$.
Therefore, `$f$ is not homotopic to zero' is not a
                            requisite to span a non-parametric minimal cone.

Although there exist uncountably many LOMSEs, for each of them
both the nonzero singular value $\la$ and the acute angle $\th$
for the associated LOS are constants depending on $(n,p,k)$ in a discrete manner (see Theorem \ref{npk}).
Consequently, we gain interesting gap phenomena for certain geometric quantities of LOSs or LOCs associated to LOMSEs, e.g. angles between normal planes and a fixed reference plane, volumes, Jordan angles and slope functions,
see Corollary \ref{cor2}.
 We remark that rigidity properties for
these quantities of compact minimal submanifolds in spheres or entire minimal graphs in Euclidean spaces have drawn attention in
many literatures \cite{b,fc,j-x,j-x-y3,j-x-y2,j-x-y4,c-l-y,p-w,j-x-y5}.

Motivated by the argument of Lawson-Osserman \cite{l-o}, we seek for analytic solutions to Dirichlet problem
for the boundary data $f_\varphi:=\varphi\cdot f$ as well.
 A good candidate (compared with \eqref{cone-graph}) turns out to be
 \begin{equation}
 F_{f,\rho}(y)=\left\{\begin{array}{cc}
\rho(|y|)f(\f{y}{|y|}) & y\neq 0\\
0 & y=0
\end{array}\right.
\end{equation}
Here $\rho$ is a smooth positive function on $(0,b)$ for some $b\in\R_+\cup \{+\infty\}$,
satisfying $\lim\limits_{r\ra 0^+}\rho=0$.
If
\begin{equation}\label{frho1}
M_{f,\rho}:=\big\{(rx,\rho(r)f(x)):x\in S^n, r\in (0,b)\big\}
\end{equation}
 is a minimal submanifold and $\rho_r(0)=0$,
 then Morrey's regularity theorem \cite{mo}
 ensures $F_{f,\rho}$ an analytic solution to the minimal surface equations through the origin.
 Since the minimality
 is invariant under rescaling,
 $F_{f,\rho_d}$
 for
  $\rho_d(r):=\f{1}{d}\,\rho(d\cdot r)$
  and $d>0$
  produce a series of minimal graphs.
 %  is also an analytic solution.
 Therefore, in the $r\rho$-plane,
 every intersection point of the graph of $\rho$ and the straight line $\rho=\varphi\cdot r$
 corresponds to an analytic solution to the Dirichlet problem for $f_\varphi$.

 In particular, when $f$ is an LOMSE, the minimal surface equations can be reduced to (\ref{ODE1}),
 a nonlinear ordinary differential equation of second order,
 equivalent to an autonomous system (\ref{ODE2}) in the $\varphi\psi$-plane
 for $\varphi:=\f{\rho}{r}$, $t:=\log r$ and $\psi:=\varphi_t$.
 With the aid of suitable barrier functions, we obtain a long-time existing bounded solution, whose orbit in the phase space
 emits from the origin - a saddle critical point and limits to $P_1(\varphi_0,0)$ - a stable critical point (see Propositions \ref{case1}-\ref{case2}).

Quite subtly, there are two dramatically different types of asymptotic behaviors
aroud $P_1$ depending on the values $(n,p,k)$ of $f$:
\begin{enumerate}
\item [(I)]
$P_1$ is a stable center when $(n,p,k)=(3,2,2), (5,4,2), (5,4,4)$ or $n\geq 7$;
\item [(II)] $P_1$ is a stable spiral point when $(n,p)=(3,2)$,
$k\geq 4$ or $(n,p)=(5,4)$, $k\geq 6$.
\end{enumerate}
                   \begin{figure}[h]
                              \begin{minipage}[c]{0.45\textwidth}
                              \includegraphics[scale=0.52]{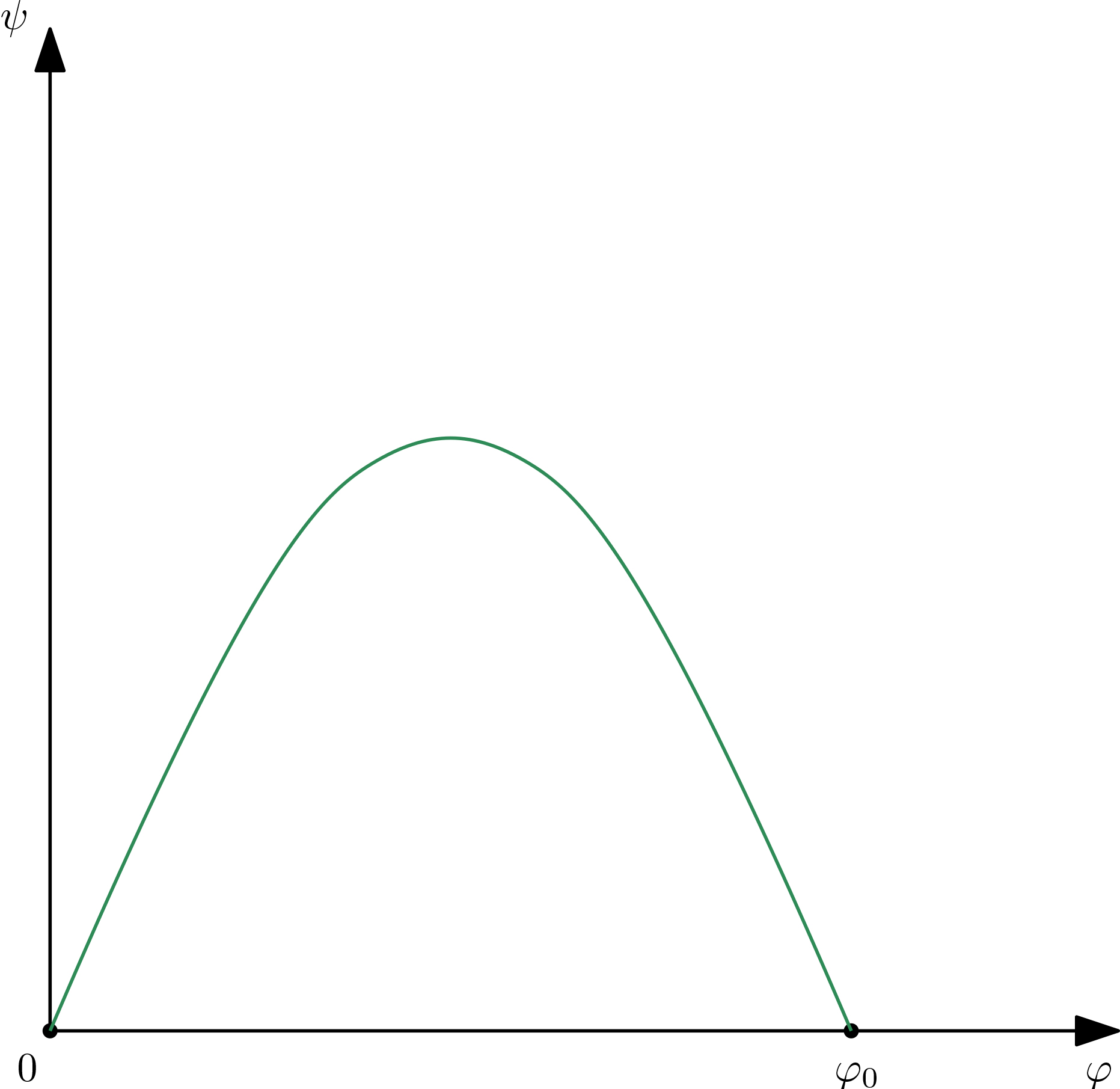}
                              \end{minipage}%
                          \begin{minipage}[c]{0.5\textwidth}
                           \includegraphics[scale=0.55]{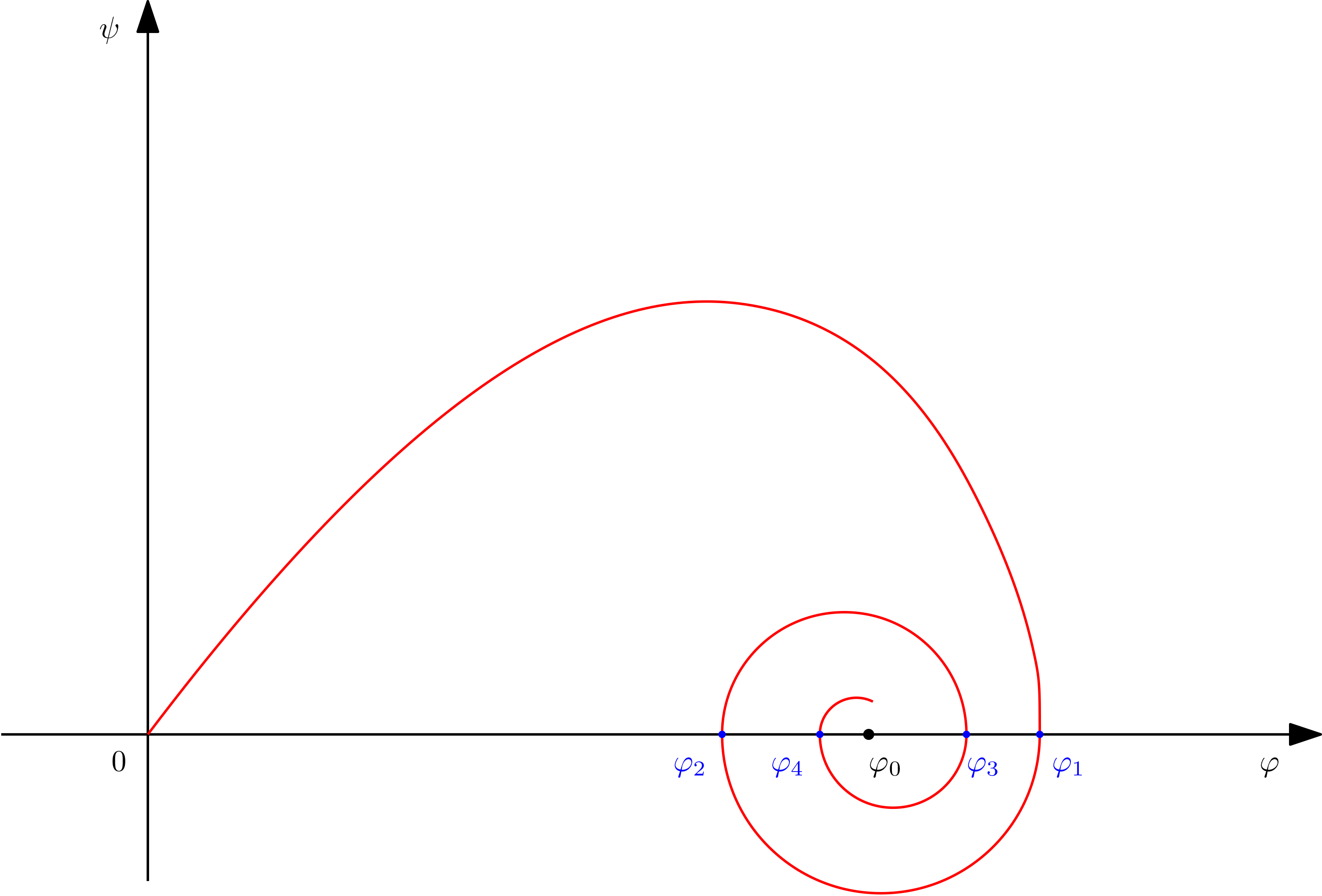}
                           \end{minipage}
                          % \caption{}
                    \end{figure}
 As a consequence, the graphs of the solutions $\rho$ to (\ref{ODE1}) are illustrated below, respectively for LOMSEs of Type (I) and Type (II).
                              $$\begin{minipage}[c]{0.5\textwidth}
                              \includegraphics[scale=0.53]{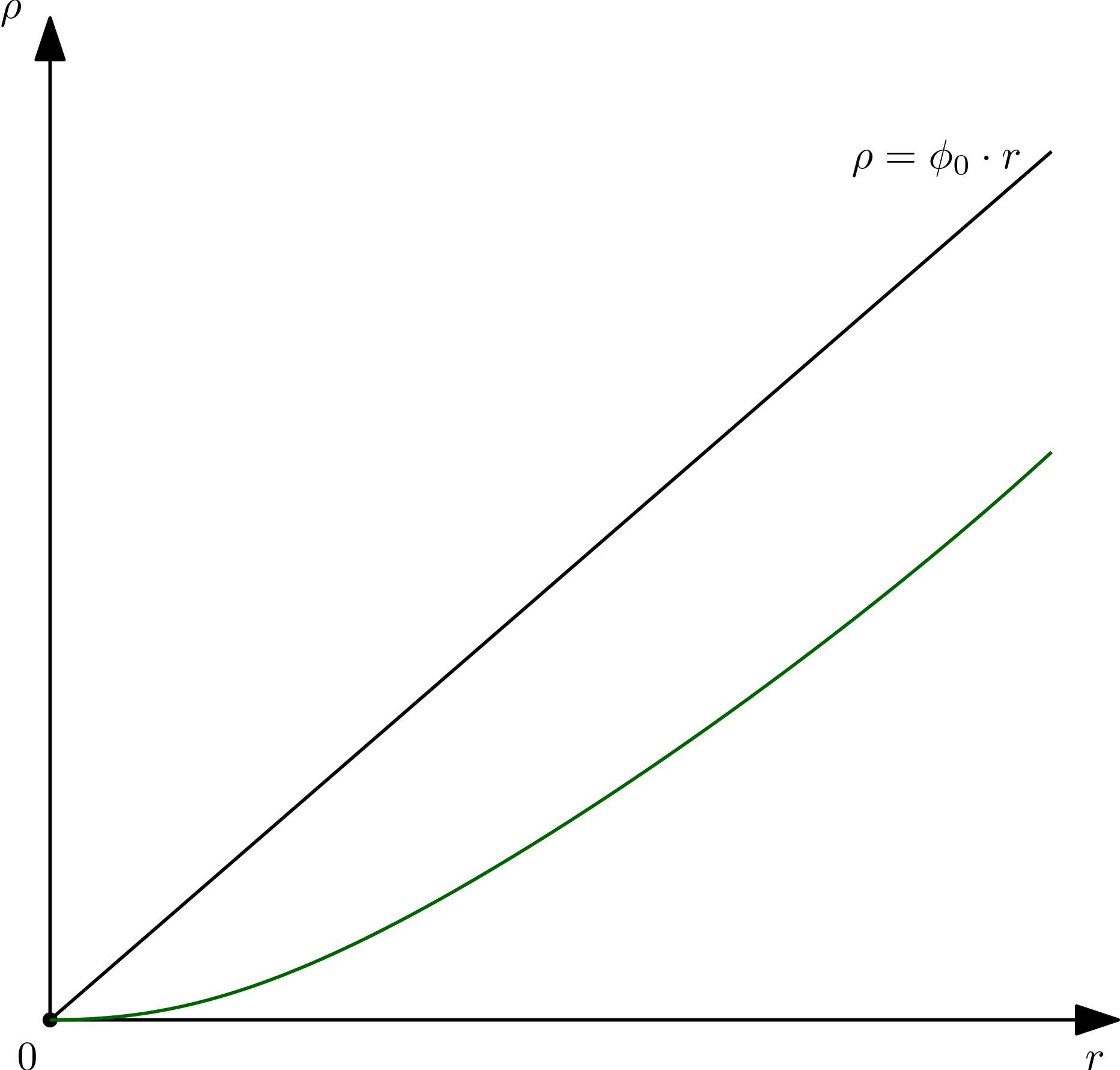}
                              \end{minipage}%
                          \begin{minipage}[c]{0.45\textwidth}
                           \includegraphics[scale=0.5]{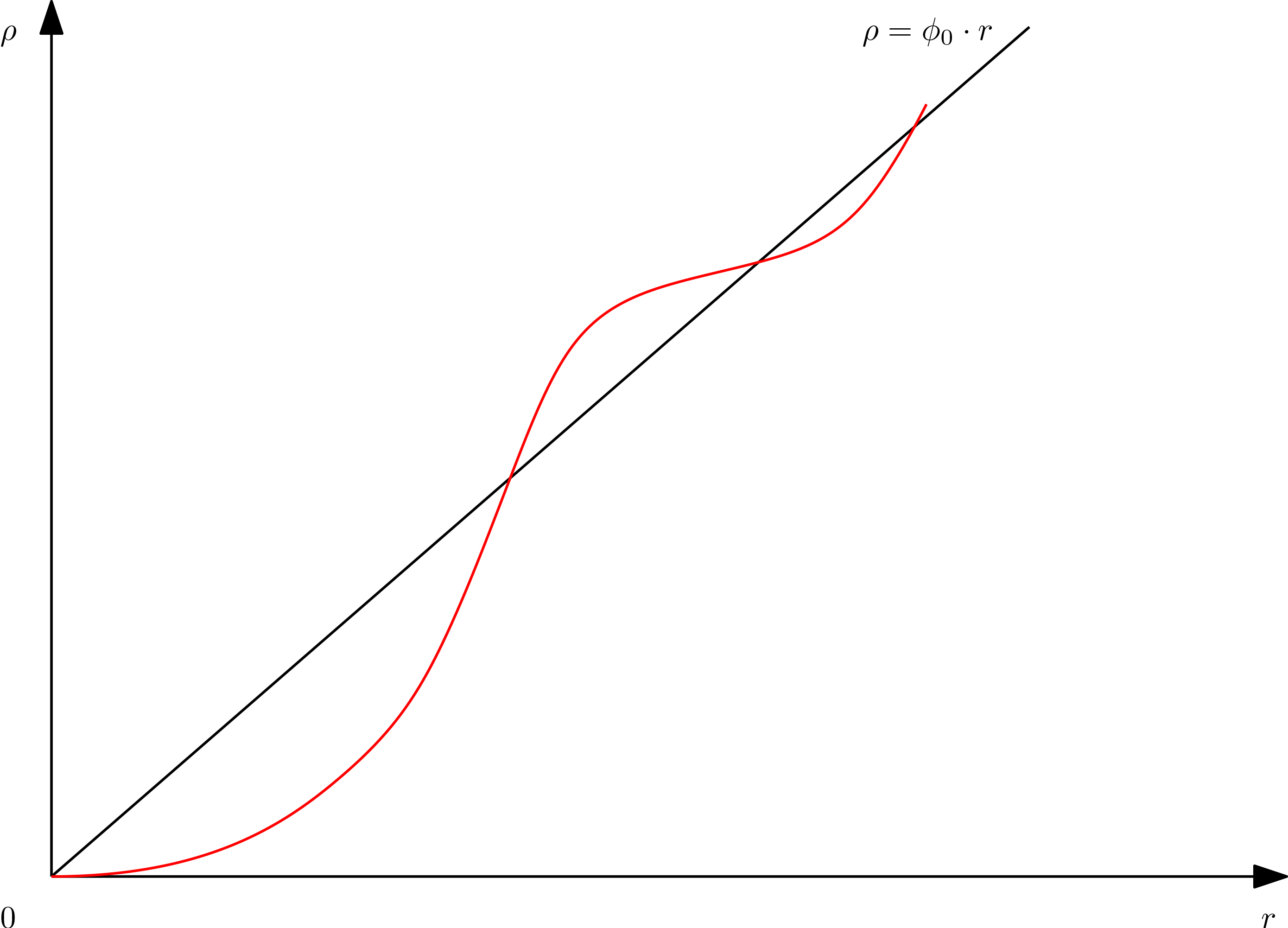}
                           \end{minipage}$$
                           %\caption{}

 Much interesting information can be read off from the above pictures:
\begin{enumerate}
                  \item [(A)]
                              For each LOMSE $f$, there exists an entire analytic minimal graph whose tangent cone at infinity is exactly the LOC
                               associated to $f$ (see Theorem \ref{graph}).
                  \item [(B)]
                              For an LOMSE $f$ of Type (II),
                              there exist infinitely many analytic solutions to the Dirichlet problem
                              for $f_{\varphi_0}$; meanwhile, it also has a singular Lipschitz solution
                              %corresponding to the truncated LOC.
                              which corresponds to the truncated LOC (see Theorem \ref{uni}).
                  \item[(C)] For Type (II),
                              although a Lipschitz solution arises for the boundary data $f_{\varphi_0}$,
                              there exists an $\epsilon>0$ such that
                                     the Dirichlet problem
                                      still has
                                     analytic solutions for $f_\varphi$
                                     whenenver
                                     $\varphi\in (\varphi_0,\varphi_0+\epsilon)$.
                  \item [(D)]
                              By the monotonicity of density for minimal submanifolds (currents) in Euclidean spaces (see \cite{fe,c-m}), LOCs associated to LOMSEs of Type (II) are all non-minimizing (see Theorem \ref{non-min}).

                  \end{enumerate}

 To the knowledge of the authors of the present paper, it seems to be the first time to have phenomena (B)-(C) observed,
                  and hard to foresee
                  the occurrence from the classical theory of partial differential equations.

By the machinery of calibrations, the LOC associated to the Hopf map from $S^3$ onto $S^2$
(i.e. the LOMSE of $(3,2,2)$-type)
was shown area-minimizing by Harvey-Lawson \cite{h-l}.
It would be interesting to consider
 whether the associated LOC is area-minimizing for an LOMSE of $(n,p,k)$-type.
 In Theorem \ref{non-min}, we establish a partial negative answer to the question. On the other hand,
 in a subsequent paper \cite{x-y-z}, we explore this subject from a different point of view and
 confirm that all LOCs associated to LOMSEs of $(n,p,2)$-type are
area-minimizing.

\bigskip\bigskip

\Section{Lawson-Osserman maps}{Lawson-Osserman maps}

\subsection{Preliminaries on harmonic maps}
Let $(M^n,g)$ and $(N^m,h)$ be Riemannian manifolds
 and $\phi$ be a smooth mapping from $M$ to $N$.
The \textit{energy desity} of $\phi$
at $x\in M$
is defined to be
\begin{equation}
e(\phi):=\f{1}{2}\sum_{i=1}^n h(\phi_* e_i,\phi_* e_i).
\end{equation}
Here $\{e_1,\cdots,e_n\}$ is an orthonormal basis of $T_xM$.
The \textit{total energy} $E(\phi)$ is
the integral of $e(\phi)$ over $M$.

Let $\td{\n}$ and $\n$ be Levi-Civita connections w.r.t. $g$ and $h$ respectively.
Then the \textit{second fundamental form} of $\phi$ is given by
\begin{equation}
B_{XY}(\phi):=\n_{\phi_* X}\phi_* Y-\phi_*(\td{\n}_X Y),
\end{equation}
whose trace under $g$ is the \textit{tensor field} of $M$
\begin{equation}
\tau(\phi):=\sum_{i=1}^n B_{e_ie_i}(\phi).
\end{equation}
 If $\tau(\phi)$ vanishes indentically,
 then $\phi$ is called a \textit{harmonic map}.
 When $B\equiv 0$,
 $\phi$ is called \textit{totally geodesic}.
 The first variation formula asserts that
 $\phi$ is harmonic if and only if it is a critical point of
 functional $E$.

For a smooth function $f:(M,g)\ra \R^n$, one can see that $\tau(f)=\De_g(f)$
where $\De_g$ is the {\it Laplace-Beltrami operator} for $g$.
Hence $f$ is harmonic if it is a {\it harmonic function} in the usual sense.

Given an isometric immersion $i: (M,g)\ra (N,h)$,
its second fundamental form can be
identified with the second fundamental form of $M$ in $N$,
and
its tensor field can be regarded as the mean curvature vector field $\mathbf{H}$.
Therefore, $i$ is  harmonic if and only if it is an isometric minimal immersion. Moreover, $i$ is
totally geodesic if and only if it is an isometric totally geodesic immersion.

For Riemannian submersions, we have the following characterization.

 \begin{pro}\label{ER} (see e.g. Proposition 1.12 of \cite{e-r})
 A Riemannian submersion $\pi:(M,g)\ra (N,h)$ is harmonic if and only if each fiber of $\pi$ is a minimal submanifold of $M$.
 \end{pro}

Let $M^n, N^m, \bar{N}$ be Riemannian manifolds, and $\phi: M^n\ra N^m$, $\bar{\phi}:N^m \ra \bar{N}$ be smooth maps.
We have the fundamental composition formula for tension fields (see Proposition 1.14 in \cite{e-r} or \S 1.4 of \cite{xin1}):
\begin{equation}\label{com1}
\tau(\bar{\phi}\circ \phi)=\bar{\phi}_*(\tau(\phi))+\sum_{j=1}^n B_{\phi_* e_j,\phi_* e_j}(\bar{\phi}).
\end{equation}
In particular, for an isometric immersion $\bar{\phi}$,
\begin{equation}\label{com2}
\tau(\bar{\phi}\circ \phi)=\tau(\phi)+\sum_{j=1}^n B(\phi_* e_j,\phi_* e_j)
\end{equation}
where $B$ is the second fundamental form of $N$ in $\bar{N}$.
\bigskip

\subsection{Necessary and sufficient conditions for LOSs}\label{NSCL}
Let $S^d\subset \R^{d+1}$ be the $d$-dimensional unit sphere,
$g_d$ the canonical metric %on $S^d$
induced by
the inclusion map $i_d: S^d\ra \R^{d+1}$, and $B_d$ the second fundamental form of $S^d$ in $\R^{d+1}$.

Given smooth $f:S^n\rightarrow S^m$ and an acute angle $\th$, let $I_{f,\th}: S^n\ra S^{n+m+1}$
\begin{equation}
I_{f,\th}(x)=(\cos\th\cdot x,\sin\th\cdot f(x))
\end{equation}
be the embedding associated to $f$ and $\th$, and $g:=I_{f,\th}^*g_{n+m+1}$.
We shall study when $I_{f,\th}$ is minimal and thus yields an LOS $M_{f,\th}$.

Let $\mathbf{X}(x)$,
$\mathbf Y_1(x)$ and $\mathbf Y_2(x)$ be the position vectors of
 $I_{f,\th}(x)$ in $\R^{n+m+2}$,
$x$ in $\R^{n+1}$ and $f(x)$ in $\R^{m+1}$ respectively.
Then
\begin{equation}\label{position}
\mathbf{X}(x)=(\cos\th\mathbf{Y}_1(x),\sin\th\mathbf{Y}_2(x)).
\end{equation}
Here
$\mathbf X$ can be viewed from two different angles. On the one hand, $\mathbf{X}$ is a vector-valued function on $(S^n,g)$ and we have $\De_g \mathbf X=\tau(\mathbf X)$.
On the other hand,
$\mathbf X=i_{n+m+1}\circ I_{f,\th}$, and consequently by the composition formula
\eqref{com2} we have
\begin{equation}\label{x2}
\aligned
\De_g \mathbf{X}&=\tau(\mathbf{X})=\tau(i_{n+m+1}\circ I_{f,\th})
=\tau(I_{f,\th})+\sum_{j=1}^n B_{n+m+1}((I_{f,\th})_* e_j,(I_{f,\th})_* e_j)\\
&=\mathbf{H}-\sum_{j=1}^n \lan (I_{f,\th})_* e_j,(I_{f,\th})_* e_j\ran\mathbf{X}
=\mathbf{H}-\sum_{j=1}^n g(e_j,e_j)\mathbf{X}=\mathbf{H}-n\mathbf{X}.
\endaligned
\end{equation}
Here $\{e_1,\cdots,e_n\}$ is an orthonormal basis of $(T_x S^n,g)$, $\lan \cdot,\cdot\ran$ the Euclidean inner product,
and
$\mathbf H$ the mean curvature field of $(S^n,g)$ in $S^{n+m+1}$.
We remark that $\mathbf H\bot \mathbf {X}$ pointwise.

Similarly,
for $\mathbf{Y_1}=i_n\circ \mathbf{Id}$ where $\mathbf{Id}$ is the identity map from $(S^n,g)$ to $(S^n,g_n)$
and $\mathbf{Y_2}=i_m\circ f$,
we derive
\begin{equation}\label{y1}
%\aligned
\De_g \mathbf{Y}_1=\tau(\mathbf{Y}_1)=\tau(i_n\circ \mathbf{Id})
%=\tau(\mathbf{Id})+\sum_{j=1}^n B_n(\mathbf{Id}_* e_j,\mathbf{Id}_* e_j)\\
%&=\tau(\mathbf{Id})+\sum_{j=1}^n B_n(e_j,e_j)
%=\tau(\mathbf{Id})-\sum_{j=1}^n \lan e_j,e_j\ran \mathbf{Y}_1
=\tau(\mathbf{Id})-2e(\mathbf{Id})\mathbf{Y}_1,
%\endaligned
\end{equation}
where $\tau(\mathbf{Id})\bot \mathbf{Y}_1$,
and
\begin{equation}\label{y2}
%\aligned
\De_g \mathbf{Y}_2=\tau(\mathbf{Y}_2)%=\tau(i_m\circ f)=\tau(f)+\sum_{j=1}^n B_m(f_* e_j,f_* e_j)\\
%&=\tau(f)-\sum_{j=1}^n \lan f_* e_j, f_* e_j\ran \mathbf{Y}_2
=\tau(f)-2e(f)\mathbf{Y}_2,
%\endaligned
\end{equation}
where $\tau(f)\bot \mathbf Y_2$.

By (\ref{y1}) and (\ref{y2}), we obtain
\begin{equation}\label{x1}
\De_g \mathbf{X}
%=(\cos\th \De_g \mathbf{Y}_1,\sin\th \De_g \mathbf{Y}_2)
=\Big(\cos\th \big(\tau(\mathbf{Id})-2e(\mathbf{Id}) \mathbf{Y}_1\big),\sin\th \big(\tau(f)-2e(f) \mathbf{Y}_2\big)\Big).
\end{equation}
Comparing \eqref{x1} and \eqref{x2} produces
\begin{equation}\label{mean}
\mathbf{H}=\Big(\cos\th \big(\tau(\mathbf{Id})-(2e(\mathbf{Id})-n) \mathbf{Y}_1\big),\sin\th \big(\tau(f)-(2e(f)-n) \mathbf{Y}_2\big)\Big).
\end{equation}
We shall employ this relationship for the characterization of LOS.

\begin{thm}\label{con}

For smooth $f:S^n\rightarrow S^m$ and $\th\in (0,\pi/2)$,
$I_{f,\th}$ is minimal (i.e., $M_{f,\th}$ is an LOS in $S^{n+m+1}$)
if and only if the following conditions hold:
\begin{enumerate}
%\item[(a)] The identity mapping $\mathbf{Id}:(S^n,g)\ra (S^n,g_n)$ is a harmonic map.
%Here and in the sequel $g_n$ is the canonical metric on the Euclidean sphere.
%\item[(b)] The energy density of $\mathbf{Id}:(S^n,g)\ra (S^n,g_n)$ is $\f{n}{2}$ everywhere.
\item[(a)] $f:(S^n,g)\ra (S^m,g_m)$ is harmonic.
%\item[(d)] The energy density of $f:(S^n,g)\ra (S^m,g_m)$ is $\f{n}{2}$ everywhere.
\item[(b)] For each $x\in S^n$ and the singular values $\la_1,\cdots,\la_n$  of
$(f_*)_x: (T_x S^n,g_n)\ra (T_{f(x)} S^m,g_m)$,
$\sum\limits_{j=1}^n \f{1}{\cos^2\th+\sin^2\th \la_j^2}=n.$
\end{enumerate}
%Conversely, if Conditions (c)-(d) hold, then $(S^n,g)$ is a minimal submanifold in $S^{n+m+1}$.

%Moreover, Conditions (b) and (d) are both equivalent to the following statement:
%\begin{enumerate}
%\item[(e)] For any $x\in S^n$, denote by $\la_1,\cdots,\la_n$ the singular values of
%$(f_*)_x: (T_x S^n,g_n)\ra (T_{f(x)} S^m,g_m)$, then
%$\sum\limits_{j=1}^n \f{1}{\cos^2\th+\sin^2\th \la_j^2}=n.$
%\end{enumerate}
\end{thm}

\begin{proof}
%$\mathbf{Y}_1$ can be view from 2 different angles. Firstly, $\mathbf{Y}_1$ is a vector-valued function on $(M,g)$.
%Secondly, $\mathbf{Y}_1=i_n \circ \mathbf{Id}$, where $\mathbf{Id}:(S^n,g)\ra (S^n,g_n)$ is the identity mapping
%and $i_n$ is the isometric embedding from $S^n$ into $\R^{n+1}$. Denote by $\De_g$ the Laplace-Beltrami operator with respect to $g$,
%then by (\ref{lem1}), at any $x\in (S^n,g)$,

%where $e_1,\cdots,e_n\in T_x S^n$ satisfy $g(e_i,e_j)=\de_{ij}$ and $\lan \cdot,\cdot\ran$ denotes the Euclidean inner product.
%Note that the second fundamental form $B$ of $(S^n,g_n)\subset \R^{n+1}$ satisfies
%$B(v,w)=-\lan v,w\ran \mathbf{Y}_1$.

%Similarly, $\mathbf{Y}_2=i_m\circ f$ and again using (\ref{lem1}) yields

%By Takahashi's theorem, $(S^n,g)$ is a minimal submanifold in $S^{n+m+1}$ if and only if $\De_g \mathbf{X}=-n X$,
%which equals to say that $\tau(\mathbf{Id})=\tau(f)=0$ and
%\begin{equation}
%\sum_{j=1}^n \lan e_j,e_j\ran=\sum_{j=1}^n \lan f_* e_j, f_* e_j\ran=n.
%\end{equation}

%On the other hand, let $F:S^n\mapsto S^{n+m+1}$ such that $\mathbf{X}(x)$ is the position vector of $F(x)$ in $\R^{n+m+2}$, then (\ref{lem1}) implies

%Here $\mathbf{H}$ denotes the mean curvature vector field of $(S^n,g)\subset S^{n+m+1}$.

 Firstly, we claim Condition (b)
has another two equivalent statements as follows:
\begin{enumerate}
\item [(c)] The energy density of $\mathbf{Id}:(S^n,g)\ra (S^n,g_n)$ is $\f{n}{2}$ everywhere.
\item [(d)] The energy density of $f:(S^n,g)\ra (S^m,g_m)$ is $\f{n}{2}$ everywhere.
\end{enumerate}
Now we give a proof of (b)$\Leftrightarrow$(c).
Due to the theory of singular value decomposition, there exists
an orthonormal basis $\{\ep_1,\cdots,\ep_n\}$ of $(T_x S^n,g_n)$, such that
\begin{equation}
\lan f_* \ep_j,f_*\ep_k\ran=\la_j^2 \de_{jk}.
\end{equation}
Set
\begin{equation}\label{basis3}
e_j:=\f{1}{\sqrt{\cos^2\th+\sin^2\th\la_j^2}}\ep_j.
\end{equation}
Then we have
\begin{equation}\aligned
g(e_j,e_k)&=\lan (I_{f,\th})_* e_j,(I_{f,\th})_* e_k\ran\\
&=\lan (\cos\th e_j,\sin\th f_* e_j),(\cos\th e_k,\sin\th f_* e_k)\ran\\
&=\cos^2\th \lan e_j,e_k\ran+\sin^2 \th\lan f_* e_j,f_* e_k\ran\\
&=\de_{jk}.
\endaligned
\end{equation}
This implies that $\{e_1,\cdots,e_n\}$ is an orthonormal basis of $(T_x S^n,g)$. Here and in the sequel,
we call such $\{\ep_1,\cdots,\ep_n\}$ and $\{e_1,\cdots,e_n\}$ the {\bf S-bases} of
$(S^n,g_n)$ and $(S^n,g)$ for $f$  (w.r.t. $\la_1,\cdots,\la_n$ and $\f{\la_1}{\sqrt{\cos^2\th+\sin^2\th \la_1^2}},\cdots,\f{\la_n}{\sqrt{\cos^2\th+\sin^2\th\la_n^2}}$ respectively). Then
$$2e(\mathbf{Id})=\sum_{j=1}^n\lan e_j,e_j\ran=\sum_{j=1}^n \f{1}{\cos^2\th+\sin^2\th\la_j^2}.$$
Therefore (b) is equivalent to (c). Also note that $\sum\limits_{j=1}^n \f{1}{\cos^2\th+\sin^2\th \la_j^2}=n$ is equivalent
to $\sum\limits_{j=1}^n \f{1}{\cos^2\th+\sin^2\th \la_j^2}=n$ for an acute angle $\th$. So  (b) and (d) are equivalent as well.

If $I_{f,\th}$ is an isometric minimal embedding, i.e. $\mathbf{H}=0$, then (\ref{mean}) implies $\tau(\mathbf{Id})=\tau(f)=0$ and $e(\mathbf{Id})=e(f)=\f{n}{2}$,
hence Conditions (a)-(b) hold.

Conversely, when Conditions (a)-(b) hold,
substituting $e(\mathbf{Id})=e(f)=\f{n}{2}$ and $\tau(f)=0$ into (\ref{mean}) implies
$\mathbf{H}=(\cos\th\cdot \tau(\mathbf{Id}),0).$
Since $\mathbf{H}\bot (I_{f,\th})_* (T_x S^n)$,
$$0=\lan \mathbf{H},(I_{f,\th})_* v\ran=\big\lan (\cos\th\cdot \tau(\mathbf{Id}),0),(v,f_*v)\big\ran=\cos\th \lan \tau(\mathbf{Id}),v\ran$$
for every $v\in T_x S^n$. Hence
\begin{equation}\label{Idh}
\tau(\mathbf{Id})=0
\end{equation}
 and moreover $\mathbf{H}=0$, i.e., $M_{f,\th}$ is an LOS in
$S^{n+m+1}$.

\end{proof}

%For any $x\in S^n$, denote by $\mc{S}_x$ the set consisting of all the distinct singular values of
%$(f_*)_x:(T_x S^n,g_n)\ra (T_{f(x)} S^m,g_m)$. The multiplicity of $\la\in \mc{S}_x$ is denoted by $m_\la$. If $\mc{S}_x$
%is independent of $x\in S^n$, and the multiplicity of each singular value is constant, then $f$ is called a smooth map
%with \textit{constant singular values (CSV)}. In this case, the set of distinct singular values is denote by $\mc{S}(f)$ and the number
%of distinct singular values is denoted by $g_f$.

\bigskip
\subsection{Characterizations of trivial LOMs}

For an isometric totally geodesic embedding $f:(S^n,g_n)\ra (S^m,g_m)$,
it is easy to see that $M_{f,\th}$ is totally geodesic in $S^{n+m+1}$ for arbitrary $\th\in (0,\pi/2)$. We call such {$f$ a} {\bf trivial LOM}.
The following characterizes trivial LOMs from the aspect of singular values.

\begin{pro}\label{g1}

For an LOM $f:S^n\ra S^m$, the followings are equivalent:

\begin{enumerate}
\item[(i)] All singular values of $(f_*)_x$ are equal at each $x$.
\item[(ii)] All singular values of $(f_*)_x$ are equal to $1$.
\item[(iii)] $f:(S^n,g_n)\ra (S^m,g_m)$ is an isometric immersion.
\item[(iv)] $f:(S^n,g_n)\ra (S^m,g_m)$ is an isometric totally geodesic embedding.
\item[(v)] For every $\th\in (0,\pi/2)$, $M_{f,\th}$ is totally geodesic.
\item[(vi)] There exists $\th\in (0,\pi/2)$, such that $M_{f,\th}$ is a totally geodesic LOS.

\end{enumerate}

\end{pro}

\begin{proof}

(i)$\Rightarrow$(ii) immediately follows from Condition (b) in Theorem \ref{con};
 (iii)$\Rightarrow$(iv) is a direct corollary
of Condition (a) in Theorem \ref{con} and the {\it Gauss equations};
%(see \S 1.1 of \cite{x}),
and the proofs of (ii)$\Rightarrow$(iii) and
(iv)$\Rightarrow$(v)$\Rightarrow$(vi)$\Rightarrow$(i) are trivial.
\end{proof}

\begin{cor}\label{cor1}
Let $f:S^n\ra S^m$ be a smooth map.
Then
\begin{itemize}
\item If $n\geq 2$ and $m=1$, then $f$ cannot be an LOM.

\item If $n\leq 2$ and $m\geq n$, then $f$ is an LOM if and only if $f$ is a trivial one.
\end{itemize}
\end{cor}

\begin{proof}
We shall study each case according to the the values of $n$ and $m$.

\textbf{Case I. }$n=1$.
In this case, $(f_*)_x$ has only one singular value. By (i) of
Proposition \ref{g1}, $f$ is an LOM if and only if $f$ is a trivial one.

\textbf{Case II. }$n\geq 2$, $m=1$.
If there were one LOM $f$, then by Theorem \ref{con}
$f:(S^n,g)\ra S^1$ is harmonic.
So is its lifting map $\td{f}:(S^n,g)\ra\R$.
But the strong maximal principle forces $\td{f}$ to be constant, and hence the same for $f$,
which contradicts (ii) of Proposition \ref{g1}.
Thus there are no LOMs in this setting.

\textbf{Case III. }$n=2, m\geq 2$.
Suppose $f$ is an LOM and $I_{f,\th}:(S^2,g)\ra (S^{3+m},g_{3+m})$ is the corresponding isometric minimal embedding.
Then by Theorem \ref{con} and \eqref{Idh}
$\mathbf{Id}: (S^2,g)\ra (S^2,g_2)$ and $f: (S^2,g)\ra (S^m,g_m)$ are both harmonic.
 It is well known that every harmonic map from a $2$-sphere
 (equipped with arbitrary metric)
 is conformal (see \S I.5 of \cite{s-y}).
For an orthonormal basis $\{e_1,e_2\}$ of $(T_x S^2,g)$,
we have
$$\lan e_1,e_1\ran =\lan e_2,e_2\ran=e(\mathbf{Id})=1,\quad \lan e_1,e_2\ran=0$$
and
$$\lan f_* e_1,f_* e_1\ran=\lan f_* e_2,f_* e_2\ran=e(f)=1,\quad \lan f_* e_1,f_* e_2\ran=0.$$
Hence $f:(S^2,g_2)\ra (S^m,g_m)$ is an isometric immersion. By (iii) of Proposition \ref{g1},
$f:(S^2,g_2)\ra (S^m,g_m)$ is a trivial LOM.
\end{proof}

\textbf{Remarks.}
\begin{itemize}
\item   The Hopf map from $S^3$ onto $S^2$ gives a nontrivial LOM.
Thus the restriction on $n$ and $m$ in Corollary \ref{cor1} is necessary and optimal.
\item 
For an LOM $f$, the corresponding $C_{f,\th}$ is flat if and only if $f$ is trivial.
Hence the second part of the above corollary follows from the rigidity theorems in \cite{c-o}, \cite{b} and \cite{fc}.
%which claim that any non-parametric minimal cone
%of dimension three or less must be flat.
However, it is unknown up to now whether or not there exists a nonflat, non-parametric minimal cone of codimension 2,
so the first part of Corollary \ref{cor1} cannot be derived from previous works.

\end{itemize}

\bigskip

\subsection{Nontrivial LOMSEs}

It can be observed that
three original LOMs,
the Hopf maps
 $H^{2m-1,m}:(S^{2m-1},g_{2m-1})\ra (S^m, g_m)$ for $m=2, 4$ and $8$,
have singular values $0$ and $2$ of multiplicities $m-1$ and $m$ pointwise.
In fact, we can gain the following structure theorem for LOMSEs. %(defined in our introduction).
%This observation motives us to study {\bf nontrivial LOMs whose nonzero singular values are equal (LOMSEs)}.

\begin{thm}\label{g2}
For smooth $f:S^n\ra S^m$, the followings are equivalent:
\begin{enumerate}
\item[(i)] $f$ is a nontrivial LOMSE, namely for each $x\in S^n$, all the nonzero singular values of $(f_*)_x$ are equal.
\item[(ii)] $f$ is an LOM, and $f$ has two constant singular values $0$ and $\la>0$
of multiplicities $(n-p)$ and $p$ respectively everywhere.
\item[(iii)] There exist a $p$-dimensional Riemannian manifold $(P,h)$ with $p<n$, a real number $\la>\sqrt{\f{n}{p}}$, a map $\pi$ from $S^n$ onto $P$ and
 a map $i$ from $P$ into $S^m$, such that $f=i\circ \pi$, $\pi:(S^n,g_n)\ra (P,h)$ is a harmonic Riemannian submersion with connected fibers, and
 $i:(P,\la^2 h)\ra (S^m,g_m)$ is an isometric minimal immersion.

\end{enumerate}
 Assume $f$ satisfies one of the above. Then $M_{f,\th}$ becomes an LOS
 exactly when
 \begin{equation}\label{th}
\th=\arccos\sqrt{\f{n-p}{n(1-\la^{-2})}}.
\end{equation}
\end{thm}

The proof of the theorem relies on the next two lemmas.

\begin{lem}\label{lem2}
Let $(\bar{N}^n,\bar{g}), (N^m,g)$ be Riemannian manifolds.
Assume $\bar{N}$ is connected and compact, and
$\phi:(\bar{N},\bar{g})\ra (N,g)$ a smooth map
with singular values 0 and 1 of multiplicities $(n-p)$ and $p$ pointwise.
Then there exist a Riemannian manifold $(P^p,h)$, a Riemannian submersion $\pi:(\bar{N},\bar{g})\ra (P,h)$ whose fibers are all connected, and an isometric
immersion $i:(P,h)\ra (N,g)$, such that $\phi=i\circ \pi$.
\end{lem}

 We save its proof to Appendix \S \ref{App1}.

\begin{lem}\label{lem3}
Let $\{K_\a:\a\in \La\}$ be a smooth foliation of $d$-dimensional submanifolds in a manifold $M^n$,
where $\La$ is an index set.
Suppose $g,\td{g}$ are Riemannian metrics on $M$, satisfying:
\begin{enumerate}
\item[(a)] $\exists$ constant $\mu>0$, such that $\td{g}|_{K_\a}=\mu g|_{K_\a}$ for all $\a\in \La$;
\item[(b)] For every $\a\in \La$, $p\in K_\a$, $v\in T_p M$ and $w\in T_p K_\a$,\\
 $\td{g}(v,w)=0$ if and only if $g(v,w)=0$.
\end{enumerate}
Then $K_\a$ is minimal in $(M,\td{g})$ if and only if it is minimal in $(M,g)$.
\end{lem}

\begin{proof}
Let $\n$ be the Levi-Civita connection for $g$.
Then we have (e.g. see \S 2.3 of \cite{dC})
\begin{equation}\label{gn}\aligned
g(\n_X Y,Z)=&\f{1}{2}\Big\{\n_X g(Y,Z)+\n_Y g(Z,X)-\n_Z g(X,Y)\\
&+g(Y,[Z,X])+g(Z,[X,Y])-g(X,[Y,Z])\Big\}
\endaligned
\end{equation}
for vector fields $X,Y,Z$ on $M$.
% (see e.g. \S 1 of \cite{w-s-y}).

With notations $B$ for the second fundamental form of $K_\a$ in $(M,g)$ and $\mathbf{H}$ the mean vector field,
we deduce from \eqref{gn} that
\begin{equation*}\aligned
&g(\mathbf{H},\nu)=\sum_{i=1}^d g(B(E_i,E_i),\nu)=\sum_{i=1}^d g(\n_{E_i}E_i,\nu)\\
=&\sum_{i=1}^d\f{1}{2}\Big\{2\n_{E_i}g(E_i,\nu)-\n_{\nu} g(E_i,E_i)+2g(E_i,[\nu,E_i])+g(\nu,[E_i,E_i])\Big\}\\
=&\sum_{i=1}^d g(E_i,[\nu,E_i]).
\endaligned
\end{equation*}
Here $\{E_1,\cdots,E_n\}$ is a local orthonormal tangent frame field on $U\subset M$,
such that for every $x\in U\cap K_\alpha$, $\{E_1(x),\cdots,E_d(x)\}$ forms an orthonormal basis
of $T_xK_\alpha$, and
in addition, $\nu$ is a vector field on $U$ that is orthogonal to leaves.

Similarly, with symbols $\td{B}$ and $\td{\mathbf{H}}$ for $\td{g}$,
we have
\begin{equation*}
%\aligned
\td{g}(\td{\mathbf{H}},\nu)=\sum_{i=1}^d \mu^{-1}\td{g}(\td{B}(E_i,E_i),\nu)=\mu^{-1}\sum_{i=1}^d \td{g}(E_i,[\nu,E_i])\
=\sum_{i=1}^d g(E_i,[\nu,E_i]).
%\endaligned
\end{equation*}
Since $\nu$ is arbitrary, $\mathbf{H}=0$ if and only if $\td{\mathbf{H}}=0$.
\end{proof}

\renewcommand{\proofname}{\bf Proof of Theorem \ref{g2}}
\begin{proof}
Let $f:S^n\ra S^m$ be an LOM,
$M_{f,\th}$ the associated LOS in $S^{n+m+1}$,
$0$ and $\la>0$ the singular values of $(f_*)_x$ of multiplicities $(n-p)$  and $p$ respectively.
Then Condition (b) of Theorem \ref{con} implies
\begin{equation}\label{la}
\la=\sqrt{\f{n\cos^2\th}{p-n\sin^2\th}}\in \left(\sqrt{\f{n}{p}},+\infty\right).
\end{equation}
Since $\la$ varies continuously in $x$, both $\la$ and $p$ have to be constant on $S^n$.
Hence (i)$\Rightarrow$(ii) and (\ref{th}) follows immediately from (\ref{la}).

(ii) means $f:(S^n,g_n)\ra (S^m,\la^{-2}g_m)$
has singular values $0$ and $1$ of multiplicities $(n-p)$ and $p$.
By Lemma \ref{lem2}, there exist a Riemannian manifold $(P^p,h)$, a Riemannian submersion $\pi:(S^n,g_n)\ra (P,h)$
with connected fibers and an isometric immersion $i:(P,h)\ra (S^m,\la^{-2}g_m)$, such that $f=i\circ \pi$. To deduce (iii), it
suffices to show both $\pi$ and $i$ are harmonic.

By Condition (a) of Theorem \ref{con}, $f:(S^n,g)\ra (S^m,g_m)$ is harmonic. So is $f:(S^n,g)\ra (S^m,\la^{-2}g_m)$.
Moreover, (\ref{com2}) leads to
$$0=\tau(f)=\tau(i\circ \pi)=\tau(\pi)+\sum_{j=1}^n B(\pi_* e_j,\pi_* e_j),$$
 where $\{e_1,\cdots,e_n\}$ can be arbitrary orthonormal basis of the tangent plane of $(S^n,g)$ at the considered point and
 $B$ the second fundamental form of the immersed $(P,h)$ in $(S^m,\la^{-1}g_m)$.
 Observe that $\tau(\pi)$ and $\sum_{j=1}^n B(\pi_* e_j,\pi_* e_j)$ are
 tangent and normal vectors to $P$ respectively.
 Therefore, $\pi: (S^n,g)\ra (P,h)$ is harmonic, and
\begin{equation}\label{B1}
\sum_{j=1}^n B(\pi_* e_j,\pi_* e_j)=0.
\end{equation}

Assume $\la_1=\cdots=\la_p=\la$ and $\la_{p+1}=\cdots=\la_n=0$.
Let $\{\ep_1,\cdots,\ep_n\}$ and $\{e_1,\cdots,e_n\}$ be S-bases of
$(T_x S^n,g_n)$ and
$(T_x S^n,g)$ for $f$ accordingly.
Then $f=i\circ \pi$ implies that
 $\{\pi_* \ep_1,\cdots,\pi_* \ep_p\}$ gives an orthonormal basis of
$(T_{\pi(x)}P,h)$ and $\pi_* \ep_i=0$ for $p+1\leq i\leq n$
(i.e., $\ep_{p+1},\cdots,\ep_n$ are fiberwise).
Hence $\sum\limits_{j=1}^p B(\pi_* \ep_j,\pi_* \ep_j)=0$
and
by \eqref{basis3},
$i:(P,h)\ra (S^m,\la^{-2}g_m)$ is an isometric minimal immersion.

Next, we show $\pi: (S^n,g_n)\ra (P,h)$ is harmonic.
By the above,
both $\pi: (S^n,g)\ra (P,\mu^2 h)$ with $\mu:=(\cos^2\th+\sin^2\th\la^{2})^{-\f{1}{2}}$
and $\pi: (S^n,g_n)\ra (P,h)$ are Riemannian submersions.
%It is well known
%that a Riemannian submersion is harmonic if and only if all fibers are minimal
%(see Proposition \ref{ER}). Therefore, for the harmonicity of $\pi: (S^n,g_n)\ra (P,h)$, it suffices to show that the minimality of each fiber under $g$
%implies its minimality for $g_n$.
Since $g$ and $g_n$ satisfy Conditions (a)-(b) of Lemma \ref{lem3},
together with Proposition \ref{ER} we gain the harmonicity of $\pi: (S^n,g_n)\ra (P,h)$
from that of $\pi$ w.r.t. $g$.
Thus, (ii)$\Rightarrow$(iii).

Finally, the proof of (iii)$\Rightarrow$(i) is quite similar to the idea of showing (ii)$\Rightarrow$(iii),
where one instead argues
that the minimality of fibers under $g_n$ also guarantees the minimality for $g$ based on Lemma \ref{lem3}.
\end{proof}

\bigskip

\subsection{LOMSEs of (n,p,k)-type} Furthermore, in conjunction with Theorem \ref{g2} and the spectrum theory of Laplacian operators,
we show the following
properties of LOMSEs.

\begin{thm}\label{npk}
Let $f:S^n\ra S^m$ be an LOMSE
with nonzero singular value $\la$ of multiplicity $p$.
Then there exists an integer $k\geq 2$,
such that:

\begin{itemize}

\item
For $i_m\circ f(x)= \big(f_1(x),\cdots,f_{m+1}(x)\big)$  in $\R^{m+1}$,
each {component} $f_i$ is a spherical harmonic function of degree $k$.
\item
$\la=\sqrt{\f{k(k+n-1)}{p}}.$
\item $M_{f,\th}$ is an LOS associated to $f$ if and only if
\begin{equation}\label{th2}
\th=\arccos\sqrt{\f{1-\f{p}{n}}{1-\f{p}{k(k+n-1)}}}.
\end{equation}
\end{itemize}
We call such $f$ an {\bf LOMSE of (n,p,k)-type}.
\end{thm}

\renewcommand{\proofname}{\it Proof.}
\begin{proof}
By Theorem \ref{g2}, there exist a Riemannian manifold $(P^p,h)$, $\pi:S^n\ra P$ and $i:P\ra S^m$, such that $f=i\circ \pi$, $\pi:(S^n,g_n)\ra (P,h)$ is
a harmonic Riemannian submersion and $i:(P,h)\ra (S^m,\la^{-2}g_m)$ is an isometric minimal immersion.

For $y\in P$, by $\mathbf Y(y)$ we mean the position vector of $i(y)$ in $\R^{m+1}$.
Then $\mathbf Y=i_m\circ i\circ \mathbf{Id}$ where
$\mathbf{Id}$ is the identity map from $(P,h)$ to $(P,\la^2 h)$.
Since $i$ is an isometric minimal immersion and $\mathbf{Id}$ a totally geodesic map,
we have $\tau(i\circ \mathbf{Id})=0$ and thereby via the composition formula \eqref{com2} obtain
\begin{equation}\label{com3}\aligned
&\De_h(\mathbf Y)=\tau(\mathbf Y)=\tau(i_m\circ i\circ \mathbf{Id})
=\tau(i\circ \mathbf{Id})+\sum_{j=1}^p B_m\left((i\circ \mathbf{Id})_*e_j,(i\circ \mathbf{Id})_*e_j\right)\\
=&-\left(\sum_{j=1}^p \lan (i\circ \mathbf{Id})_*e_j,(i\circ \mathbf{Id})_*e_j\ran\right)\mathbf Y=-\left(\sum_{j=1}^p \la^2 h(e_j,e_j)\right)\mathbf Y
=-\la^2 p\cdot\mathbf Y.
\endaligned
\end{equation}
Here $\{e_1,\cdots,e_p\}$ is an orthonormal basis of $(T_y P,h)$.
For $\big(h_1(y),\cdots,h_{m+1}(y)\big):=\mathbf{Y}(y)$,
\eqref{com3} states precisely
\begin{equation}\label{lap_h}
\De_h(h_j)=-\la^2 p\cdot h_j,\ \ \ \ \text{for } 1\leq j\leq m+1.
\end{equation}

Coupling \eqref{com1} with \eqref{lap_h}, we get
\begin{equation}\aligned
\De_{g_n}(h_j\circ \pi)&=\tau(h_j\circ \pi)=(h_j)_*(\tau(\pi))+\sum_{j=1}^n B_{\pi_* \ep_j,\pi_* \ep_j}(h_j)\\
&=\sum_{j=1}^n \text{Hess}_h(h_j)(\pi_* \ep_j,\pi_* \ep_j)=\De_h(h_j)\circ \pi\\
&=-\la^2 p(h_j\circ \pi),
\endaligned
\end{equation}
where $\{\ep_1,\cdots,\ep_n\}$ is an orthonormal basis of $(T_x S^n,g_n)$,
such that $\{\pi_* \ep_1,\cdots,\pi_* \ep_p\}$ forms an orthonormal basis
of $(T_{\pi(x)}P,h)$ and $\pi_* \ep_{p+1}=\cdots=\pi_* \ep_n=0$.
In other words,
\begin{equation}
\De_{g_n}f_j=-\la^2 p\cdot f_j\qquad \forall 1\leq j\leq m+1.
\end{equation}

The theory of eigenvalues of the Laplacian on Euclidean spheres
confirms the existence of a positive integer $k$ so that
$f_j$ is a spherical harmonic function of degree $k$ (see \S II.4 of \cite{c})
and $\la^2 p=k(k+n-1)$, i.e.,
\begin{equation}\label{sing1}
\la=\sqrt{\f{k(k+n-1)}{p}}.
\end{equation}
Moreover, $\la>\sqrt{\f{n}{p}}$ forces $k\geq 2$.
Finally, \eqref{sing1} and \eqref{th} give (\ref{th2}).
\end{proof}

Based on Theorem \ref{npk}, several geometric quantities of LOSs or LOCs for LOMSEs of $(n,p,k)$-type
can be expressed explicitly. See Appendix \S \ref{App2} for the details.

\begin{cor}\label{cor2}
Let $f$ be an LOMSE of $(n,p,k)$-type, $M_{f,\th}$ and $C_{f,\th}$ the corresponding LOS and LOC. Then
\begin{enumerate}
\item[(A)] All normal planes of $M_{f,\th}$ make a constant acute angle $\a_{n,p,k}$ to a preferred reference plane $Q_0$ (see \S\ref{App2}), with
\begin{equation}\label{angle}
\cos\a_{n,p,k}=\sqrt{\f{1-\f{p}{n}}{1-\f{p}{k(k+n-1)}}}\cdot\left(\f{n-p}{k(k+n-1)-p}\right)^{\f{p}{2}}.
\end{equation}
\item[(B)] The volume of $M_{f,\th}$ is
\begin{equation}\label{volume}
V_{n,p,k}=\left(\f{k(k+n-1)}{n}\right)^{\f{p}{2}}\left(\f{1-\f{p}{n}}{1-\f{p}{k(k+n-1)}}\right)^{\f{n-p}{2}}\om_n,
\end{equation}
where $\om_n$ stands for the volume of $n$-dimensional unit Euclidean sphere.
\item[(C)] $C_{f,\th}$ is an entire minimal graph with constant Jordan angles relative to $Q_0$. The tangent Jordan angles of
$C_{f,\th}$ are
\begin{equation}\label{JA}
\arccos\sqrt{\f{n-p}{k(k+n-1)-p}},\quad\arccos\sqrt{\f{1-\f{p}{n}}{1-\f{p}{k(k+n-1)}}},\quad 0,
\end{equation}
of multiplicities $p,1,n-p$ respectively. The slope function of $C_{f,\th}$ is identically equal to $W_{n,p,k}:=\sec \a_{n,p,k}$.
\end{enumerate}

\end{cor}

\textbf{Remarks.}
\begin{itemize}
\item By Theorem \ref{npk2} and (A) of Corollary \ref{cor2}, the three original LOMs are LOMSEs of $(2m-1,2m,2)$-type, for $m=2,4,8$.
They are compact minimal submanifolds in spheres whose normal planes make constant angles $\a_{3,2,2}$, $\a_{7,4,2}$ and $\a_{15,8,2}$ to preferred reference planes respectively,
with $\cos\a_{3,2,2}=\f{1}{9}$, $\cos\a_{7,4,2}=\f{1}{8\sqrt{7}}$
and $\cos\a_{15,8,2}=7^4\cdot 2^{-11}\cdot 3^{-5}\sqrt{\f{7}{5}}$, as pointed out by Lawson-Osserman \cite{l-o}.
\item It was shown by E. Calabi \cite{ca} that the area of all compact minimal surfaces in spheres that is homeomorphic to $S^2$ has to be an integral of $2\pi$.
For higher dimensional cases, the existence of the gap between the volume of the totally geodesic subsphere and the volumes of other compact minimal submanifolds in spheres was discovered by Cheng-Li-Yau \cite{c-l-y}. It is natural for us to ask whether the volumes of compact minimal submanifolds
in a Euclidean sphere take values in a discrete set. Besides the work of Perdomo-Wei
\cite{p-w} on minimal rotational hypersurfaces, (B) gives a positive evidence to support this
conjecture from another viewpoint.
\item The \textit{Jordan angles} between two $d_2$-planes $P$ and $Q$ in $\R^{d_1+d_2}$ are the
critical values of the angles between the nonzero vectors $u$ in $P$ and their orthogonal projection $u^*$ in $Q$. This concept was first introduced by
C. Jordan \cite{j}. Given a submanifold $M^{d_1}$ in $\R^{d_1+d_2}$, if the Jordan angles between all normal planes of $M$ and a fixed reference plane
are constant, $M$ is called a {\it submanifold with constant Jordan angles (CJA)} (see \cite{j-x-y5}). (C) tells that, although
the LOCs derived from LOMSEs (which are all submanifolds
with CJA) are uncountable infinite (see Theorem \ref{npk2} and the remarks on it), their constant slope functions take values in a discrete set. This gives a partial positive answer to Problem 1.1 in \cite{j-x-y5}.
\end{itemize}

%\textcolor{red}{We recall some terminologies (see \cite{c-w}) of immersions into the Euclidean sphere. A minimal immersion $f:S^n\ra S^m\subset\R^{m+1}$ is \emph{full} if $f(S^n)$ is not contained in a hyperplane of $\R^{m+1}$, and two such immersions $f_1,f_2$ are \emph{equivalent} if there exists an isometry $A$ of $S^m$ such that $f_2=A\circ f_1$. Base on the classification of Riemannian submersions from Euclidean spheres and minimal immersions from projective spaces into the Euclidean sphere, we obtain the following classification theorem.}

Let $f_1:S^{n}\ra S^{m_1}$, $f_2:S^{n}\ra S^{m_2}$ be nontrivial LOMs and $m_1\leq m_2$. If
 there exist an isometry $\chi:(S^n,g_n)\ra (S^n,g_n)$ and a totally geodesic isometric embedding $\psi:(S^{m_1},g_{m_1})\ra (S^{m_2},g_{m_2})$,
 such that the following diagram commutes
$$\CD
 S^n @>\chi>> S^n  \\
 @Vf_1VV     @VVf_2 V \\
 S^{m_1}  @>\psi>> S^{m_2}
\endCD$$
then $f_1$ and $f_2$ are said to be \textit{equivalent}.
By the virtue of structure theorems on Riemannian submersions from Euclidean spheres and minimal immersions into Euclidean spheres,
we obtain a classification theorem for LOMSEs.

\begin{thm}\label{npk2}
Let $\mathcal{F}_{n,p,k}$ be the set of all equivalence classes of $(n,p,k)$-type LOMSEs.
Then $\mathcal{F}_{n,p,k}$ is nonempty if and only if $k$ is a positive even integer and $(n,p)=(15,8)$, $(2l+1,2l)$ or $(4l+3,4l)$ for some positive integer $l$.
Moreover,

\begin{itemize}

\item If $(n,p)=(2l+1,2l)$, there exists a $1:1$ correspondence between $\mathcal{F}_{2l+1,2l,k}$ and the set of equivalence classes of full isometric minimal immersions (see \cite{c-w} for definitions of `equivalence' and `full')
of $(\mathbb{CP}^{l},\frac{k(k+2l)}{2l}g_{FS})$ into unit Euclidean spheres,
where $g_{FS}$ is the Fubini-Study metric.
%$m\leq l(l+k)\left(\f{l(l+1)\cdots (l+\f{k}{2}-1)}{(\f{k}{2})!}\right)^2-1$.

\item If $(n,p)=(4l+3,4l)$, there exists a $1:1$ correspondence between $\mathcal{F}_{4l+1,4l,k}$ and the set of equivalence classes of full isometric minimal immersions
of $(\mathbb{HP}^{l},\frac{k(k+4l+2)}{4l}g_{ST})$ into unit Euclidean spheres, where  $g_{ST}$  is the standard metric on $\mathbb{HP}^l$ (see \S 3.2 of \cite{b-f-l-p-p} for details).
%$m\leq \frac{2(k+2l+1)(k+4l+1)!}{k!(4l+2)!}-1$.

\item If $(n,p)=(15,8)$, there exists a $1:1$ correspondence between $\mathcal{F}_{15,8,k}$ and the set of equivalence classes of full isometric minimal immersions of
$(S^{8},\frac{k(k+14)}{32}g_{8})$ into unit Euclidean spheres.
%$(S^{m},g_m)$, where $m\leq \frac{2(k+7)(k+13)!}{k!14!}-1$.
\end{itemize}
\end{thm}

%\textbf{Remarks. }
%\begin{itemize}
%\item An immersion $i: P\ra S^m\subset \R^{m+1}$ is said to be {\it full} if $i(P)$ is not contained in any hyperplane of $\R^{m+1}$, and
%two immersions are said to be {\it equivalent} if they differ by an isometry of $S^m$, see \cite{c-w}.

%\item The {\it standard metric} on $\Bbb{HP}^l$ is an analog of the Fubini-Study metric on $\Bbb{CP}^l$, see \S 3.2 of \cite{b-f-l-p-p}
%for details.
%\end{itemize}

\begin{proof}
 By Theorem \ref{g2}, an LOMSE $f$ can be written as $f=i\circ \pi$, where
 $\pi:(S^n,g_n)\longrightarrow (P,h)$ is a harmonic Riemannian submersion of connected fibers and
 $i:(P,\lambda^2h)\longrightarrow(S^m,g_m)$ is an isometric minimal immersion. B. Wilking's classification theorem \cite{w}
states
that all Riemannian submersions from unit Euclidean spheres with connected fibers are exactly the Hopf fibrations:
$(S^{2l+1},g_{2l+1})\longrightarrow (\mathbb{CP}^l,g_{FS})$,
$(S^{4l+3},g_{4l+3})\longrightarrow (\mathbb{HP}^l,g_{ST})$ and
$(S^{15},g_{15})\longrightarrow(S^8,\frac{1}{4}g_8)$.
%Here, $g_{FS}$ (resp. $g_{ST}$) is the Fubini-Study (resp. standard ) metric on $\mathbb{CP}^l$ (resp. $\mathbb{HP}^l$).
Therefore,
the set of all equivalence classes of $(n,p,k)$-type LOMSEs corresponds to the set of equivalence classes of full isometric minimal immersions
from $(\mathbb{CP}^l,\frac{k(k+2l)}{2l}g_{FS})$ (when $(n,p)=(2l+1,2l)$),
$(\mathbb{HP}^l,\frac{k(k+4l+2)}{4l}g_{ST})$ (when $(n,p)=(4l+3,4l)$) or $(S^8,\frac{k(k+14)}{32}g_{8})$ (when $(n,p)=(15,8)$), into unit Euclidean spheres. %\textcolor{red}{To show $k$ is even, we notice that the coordinate functions of $f$ in $\R^{m+1}$ are spherical harmonic polynomials of degree $k$ by Theorem \ref{npk}, then it is follows from the facts $f=i\circ\pi$ and $\pi(x)=\pi(-x)$, $x\in S^n$.}
By Theorem \ref{npk}, the coordinate functions of $f$ in $\R^{m+1}$ are all spherical harmonic polynomials of degree $k$. Since
$f=i\circ \pi$ and $\pi(x)=\pi(-x),x\in S^n$, $k$ has to be even for $\mathcal{F}_{n,p,k}$ being nonempty.

Given $(n,p)=(2l+1,2l)$ and $k=2\kappa$, where $l,\k\in \Bbb{Z}^+$, we will show $\mathcal{F}_{n,p,k}$ is nonempty. The similar argument holds for the other cases. Let $V_{\kappa}$ be the eigenspace of the Laplace-Beltrami operator of $(\mathbb{CP}^l,\frac{k(k+2l)}{2l}g_{FS})$ corresponding to the $\kappa$-th eigenvalue. It is known (see e.g. \S III.C of \cite{b-g-m}) that $V_\kappa$ is nonempty, and the elements in $V_\kappa$ are $S^1$-invariant spherical polynomials of degree $2\kappa$ on $S^{2l+1}$. Choosing an orthonormal basis $\{f_1,\ldots, f_{m+1}\}$ of $V_\kappa$ w.r.t. the $L^2$-inner product of a normalized measure defined in \cite{c-w,u}, then by Takahashi's Theorem \cite{ta} we know that the isometric immersion $i:(\mathbb{CP}^l,\frac{k(k+2l)}{2l}g_{FS})\longrightarrow S^m$, $x\mapsto (f_1(x),\ldots, f_{m+1}(x))$ is minimal.
This is called the \textit{standard minimal immersion} in \cite{c-w,u}.
Combining Theorems \ref{g2} and \ref{npk} implies that $f:=i\circ \pi$ is a LOMSE of $(2l+1,2l,k)$-type.
Such an LOMSE will be called a \textbf{standard
LOMSE} in the sequel. This completes the proof.
\end{proof}

{\bf Remarks.}
\begin{itemize}
\item
For $m=2$, $4$ or $8$, the Hopf map $H^{2m-1,m}$ is just the standard LOMSE of $(2m-1,m,2)$-type. From such observation, we construct all the standard LOMSEs of $(2l+1,2l,2)$, $(4l+3,4l,2)$-type in
\cite{x-y-z}.
By the rigidity theorems proved by E. Calabi \cite{ca}, do Carmo-Wallach \cite{c-w}, N. Wallach \cite{wa}, K. Mashimo \cite{ma1,ma2}
and Ohnita \cite{oh},
our construction exhausts all LOMSEs of $(2l+1,2l,2)$, $(4l+3,4l,2)$-type.
We also show that their corresponding LOCs are area-minimizing therein.

\item
%Sometimes, the dimension of the modular space of $\mc{F}_{n,p,k}$ is very large.
%In fact,
In conjunction with Theorem \ref{npk2} and the structure theorems for minimal immersions
from symmetric spaces into spheres done by do Carmo-Wallach \cite{c-w}, Wallach \cite{wa} and Urakawa \cite{u},
$\mc{F}_{n,p,k}$ can be smoothly parameterized by a convex body $L$ in a vector space $W_2$.
Based on the works of do Carmo-Wallach \cite{c-w}, G. Toth \cite{to} and H. Urakawa  \cite{u}, some partial estimates of $\dim W_2$ 
are given as follows:
 $\dim W_2\geq 18$ for $(n,p)=(7,4)$ or $(15,8)$ and $k\geq 8$;
 $\dim W_2\geq 91$ for $(n,p)=(2l+1,2l)$, $l\geq 2$, $k\geq 8$ and $\dim W_2\geq 29007$ for $(n,p)=(11,8)$, $k\geq 8$.

\end{itemize}

\bigskip\bigskip

\Section{On Dirichlet problems related to LOMSEs}{On Dirichlet problems related to LOMSEs}

\subsection{Necessary and sufficient conditions for minimal graphs}

Given smooth $f:S^n\ra S^m$ and smooth $\rho:U\subset (0,\infty)\ra \R$,
in this subsection we shall focus on
the question
when the submanifold $M_{f,\rho}$ in $\R^{n+m+2}$ of form \eqref{frho1}
 is minimal.

Let $g$ be the induced metric on $M_{f,\rho}$
and $h_r:=I_r^*g$ for $r\in U$ where
$I_r:S^n\ra M_{f,\rho}$
%$I_r$ sends
\begin{equation}
x\mapsto (rx,\rho(r)f(x)).
\end{equation}
Then $(S^n,h_r)$ is an isometric embedded Riemannian submanifold in $(M_{f,\rho},g)$.
There are two smooth functions $(rx,\rho(r)f(x))\mapsto r$ and $(rx,\rho(r)f(x))\mapsto \rho(r)$ on $M_{f,\rho}$,
and we name them briefly
$r$ and $\rho$.
From now on
we use the symbol $\n$ for the Levi-Civita connection on $(M_{f,\rho},g)$.
Then obviously %$r$ and $\rho$ are constant in $I_{r_0}(S^n)$ and hence
$\n_v r=\n_v \rho=0$ for any $v\in TI_{r_0}(S^n)$.

We derive the following minimality characterization for $M_{f,\rho}$ in terms of $r$ and $\rho$.

\begin{thm}\label{min_graph1}
%Assume $f:S^n\ra S^m$ and $\rho$ is a smooth positive function on $U\subset (0,\infty)$, then
Assume the above function $\rho>0$.
Then $(M_{f,\rho},g)$ is minimal in $\R^{n+m+2}$ if and only if the following two conditions hold:
\begin{enumerate}
\item[(a)]
For each $r\in U$,
$f:(S^n,h_r)\ra (S^m,g_m)$ is harmonic.
\item[(b)] For each $r\in U$,
$\De_g \rho-2\rho\cdot e(f)=0$ pointwise in $I_r(S^n)$ where $e(f)$ is the energy density of $f:(S^n,h_r)\ra (S^m,g_m)$.
\end{enumerate}

Moreover, Condition (b) has an equivalent description in terms of
singular values $\la_1,\cdots,\la_n$ of $(f_*)_x: (T_x S^n,g_n)\ra
(T_{f(x)}S^m,g_m)$, and that is
\begin{equation}\label{ODE}
\f{\rho_{rr}}{1+\rho_r^2}+\sum_{i=1}^n \f{\f{\rho_r}{r}-\f{\la_i^2 \rho}{r^2}}{1+\f{\la_i^2\rho^2}{r^2}}=0.
\end{equation}

\end{thm}

\begin{proof}
Define $\mathbf{X}:U\times S^n\ra \R^{n+m+2}$ by
$$(r,x)\mapsto (r\mathbf{Y}_1(x),\rho(r)\mathbf{Y}_2(x))$$
where $\mathbf Y_1(x)\in \R^{n+1}$ and $\mathbf Y_2(x)\in \R^{m+1}$ are position vectors
of $x$ and $f(x)$ respectively.
Then $\mathbf{X}$ is the position function of $M_{f,\rho}$. The tangent plane at $\mathbf{X}(r,x)$ is spanned
 by
 %$\pr$
%and $E_i$ ($1\leq i\leq n$), where
$$\pr:=(\mathbf{Y}_1(x),\rho_r \mathbf{Y}_2(x))$$
and
$$E_i:=(r\ep_i,\rho(r)f_* \ep_i)$$
determined by a basis $\{\ep_1,\cdots,\ep_n\}$ of $T_x S^n$. Moreover,
\begin{equation}
\aligned
\lan \pr,E_i\ran&=\big\lan (\mathbf{Y}_1(x),\rho_r \mathbf{Y}_2(x)),(r\ep_i,\rho(r)f_* \ep_i)\big\ran\\
&=r\lan \mathbf{Y}_1(x),\ep_i\ran+\rho_r\rho\lan \mathbf{Y}_2(x),f_* \ep_i\ran=0
\endaligned
\end{equation}
and
\begin{equation}\label{dr}
\lan \pr,\pr\ran=1+\rho_r^2.
\end{equation}

Noting that the mean curvature vector field on $M_{f,\rho}$
\begin{equation}\label{MeanCur}
\mb{H}=\De_g \mb{X}=\big(\De_g(r\mb{Y}_1),\De_g(\rho(r)\mb{Y}_2)\big),
\end{equation}
we do the following calculations in understanding the second component.

Viewing $\mathbf{Y}_2(x)$ as a vector-valued function
independent of $r$
on $(M_{f,\rho},g)$,
we get
$$
\Hess\mathbf{Y}_2(\pr,\pr)=\n_{\pr}\n_{\pr}\mathbf{Y}_2-(\n_{\pr}\pr)\mathbf{Y}_2=-(\n_{\pr}\pr)\mathbf{Y}_2.$$
By
$
\lan \n_{\pr}\pr,E_i\ran=\lan\n_{\pr}(\mathbf{Y}_1(x),\rho_r\mathbf{Y}_2(x)), E_i\ran
=\big\lan (0,\rho_{rr} \mathbf{Y}_2(x)),(r\ep_i,\rho(r) f_* \ep_i)\ran=0
$
for $1\leq i\leq n$,
$\n_{\pr}\pr$ is parallel to $\pr$ and hence $\Hess \mathbf{Y}_2(\pr,\pr)=0$.
Moreover,
in the Riemannian submanifold $(S^n,h_r)$ in $(M_{f,\rho},g)$,
we have
\begin{equation}
\De_g \mathbf{Y}_2=\De_{h_r}\mathbf{Y}_2+\lan \pr,\pr\ran^{-1}\Hess \mathbf{Y}_2(\pr,\pr)=\De_{h_r}\mathbf{Y}_2.
\end{equation}
Hence, as in \S \ref{NSCL}, for $f:(S^n,h_r)\ra (S^m,g_m)$ we gain
\begin{equation}
\De_g \mb{Y}_2=\De_{h_r}\mathbf{Y}_2=\tau(\mathbf{Y}_2)=\tau(f)-2e(f)\mb{Y}_2,
\end{equation}
and further,
\begin{equation}\label{La}
\aligned
&\De_g(\rho(r)\mb{Y}_2)\\
=&(\De_g\rho) \mb{Y}_2+\rho \De_g \mb{Y}_2+2\lan \pr,\pr\ran^{-1}\rho_r \n_{\pr}\mb{Y}_2+\sum_{i,j}g^{ij}\n_{E_i}\rho\n_{E_j}\mathbf Y_2\\
=&\rho\cdot\tau(f)+(\De_g \rho-2\rho\cdot e(f))\mb{Y}_2,
\endaligned
\end{equation}
where $(g^{ij})$ is the inverse matrix of $(g_{ij}):=\big(\lan E_i,E_j\ran\big)$.

Therefore $\mb{H}=0$ implies
$\tau(f)=0$ and $\De_g \rho-2\rho\cdot e(f)=0$.

Conversely, $\tau(f)=0$, and $\De_g \rho-2\rho\cdot e(f)=0$ lead to
$\mb{H}=(\De_g(r\mb{Y}_1),0)$.
Since
$$0=\lan \mb{H},\pr\ran=\big\lan (\De_g(r\mb{Y}_1),0),(\mb{Y}_1,\rho_r\mb{Y}_2)\big\ran=\lan \De_g(r\mb{Y}_1),\mb{Y}_1\ran$$
and
$$0=\lan \mb{H},E_i\ran=\big\lan (\De_g(r\mb{Y}_1),0),(r\ep_i,\rho f_* \ep_i)\big\ran=r\lan \De_g(r\mb{Y}_1),\ep_i\ran\quad \forall 1\leq i\leq n,$$
we deduce $\De_g(r\mb{Y}_1)=0$ and thus $\mb{H}=0$.

To exhibit the congruence of Condition (b) and \eqref{ODE},
let us figure out explicit expressions of $e(f)$ and $\De_g \rho$.
For an S-basis $\{\ep_1,\cdots,\ep_n\}$ of $(T_x S^n,g_n)$ subject to $\la_1,\cdots,\la_n$,
\begin{equation}
\lan E_i,E_j\ran=\lan (r\ep_i,\rho_r f_* \ep_i),(r\ep_j,\rho f_* \ep_j)\ran=(r^2+\rho^2 \la_i^2)\de_{ij}
\end{equation}
and
\begin{equation}\label{den2}
2e(f)=\sum_{i=1}^n \f{\lan f_* \ep_i,f_* \ep_i\ran}{h_r(\ep_i,\ep_i)}=\sum_{i=1}^n \f{\lan f_* \ep_i,f_* \ep_i\ran}{\lan E_i,E_i\ran}=\sum_{i=1}^n \f{\la_i^2}{r^2+\rho^2\la_i^2}.
\end{equation}
On the other hand,
%for $\De_g \rho(r)$
we have
\begin{equation}\label{La1}\aligned
\De_g \rho(r)&=\rho_r \De_g r+\rho_{rr}|\text{grad}_g r|^2\\
&=\rho_r\left(\f{\text{Hess}_g r(\pr,\pr)}{\lan \pr,\pr\ran}+\sum_{i=1}^n \f{\text{Hess}_g r(E_i,E_i)}{\lan E_i,E_i\ran}\right)+\rho_{rr}|\text{grad}_g r|^2
\endaligned
\end{equation}
where $\text{Hess}_g$ and $\text{grad}_g$
are the Hessian operator and the gradient operator respectively w.r.t. $g$.
Through the computations,
\begin{equation}\label{grad}
|\text{grad}_g r|^2=\lan \pr,\pr\ran^{-1}=\f{1}{1+\rho_r^2},
\end{equation}
\begin{equation}\label{hess1}\aligned
&\text{Hess}_g r(\pr,\pr)=\n_{\pr}dr(\pr)-dr(\n_{\pr}\pr)\\
=&-\f{\lan\n_{\pr}\pr,\pr\ran}{\lan \pr,\pr\ran}dr(\pr)=-\f{\n_{\pr} \lan \pr,\pr\ran}{2 \lan\pr,\pr\ran}
%=-\f{(1+\rho_r^2)_r}{2(1+\rho_r^2)}
=-\f{\rho_r\rho_{rr}}{1+\rho_r^2}
\endaligned
\end{equation}
and
\begin{equation}\label{hess2}\aligned
&\text{Hess}_g r(E_i,E_i)=\n_{E_i}dr(E_i)-dr(\n_{E_i}E_i)\\
=&-\f{\lan \n_{E_i}E_i,\pr\ran}{\lan \pr,\pr\ran}dr(\pr)=\f{\lan \n_{E_i}\pr,E_i\ran}{\lan \pr,\pr\ran}
%=\f{\lan (e_i,\rho_r f_* e_i),(re_i,\rho f_* e_i)\ran}{1+\rho_r^2}
=\f{r+\rho\rho_r\la_i^2}{1+\rho_r^2}.
\endaligned
\end{equation}
Then \eqref{La1} is simplified to be
\begin{equation}\label{La3}
%\aligned
\De_g \rho
%&=-\f{\rho_r^2\rho_{rr}}{(1+\rho_r^2)^2}+\sum_{i=1}^n\f{\rho_r(r+\rho\rho_r \la_i^2)}{(1+\rho_r^2)(r^2+\rho^2\la_i^2)}+\f{\rho_{rr}}{1+\rho_r^2}\\
%&
=\f{\rho_{rr}}{(1+\rho_r^2)^2}+\sum_{i=1}^n\f{\rho_r(r+\rho\rho_r \la_i^2)}{(1+\rho_r^2)(r^2+\rho^2\la_i^2)}.
%\endaligned
\end{equation}

By \eqref{den2} and \eqref{La3},
Condition (b) becomes \eqref{ODE}.
\end{proof}

\textbf{Remark. }
Let $\rho:U\ra \R$ be smooth (not requiring $\rho>0$) so that $M_{f,\rho}$ is minimal.
Set
$Z=\{r:\rho(r)=0\}$.
Then by \eqref{La}, $f:(S^n,h_r)\ra (S^m,g_m)$ is harmonic for $r\in U- Z$.
If $Z$ has interior points, the analyticity forces $\rho\equiv 0$.
For $\rho\not\equiv 0$,
since the tension field is smoothly depending on the metric,
 the harmonicity of $f:(S^n,h_r)\ra (S^m,g_m)$ holds for $r\in Z$.
Therefore, `$\rho$ is positive' in the theorem can be replaced by `$\rho$ is not identically vanishing'.

\bigskip

In the remaining part of the present paper, we shall focus on the case of LOMSE.
We will first establish a simple version of Theorem \ref{min_graph1} for LOMSE
and then obtain its several interesting applications.

\subsection{Entire minimal graphs associated to LOMSEs}\label{EMG}

Recall that for an LOMSE $f:S^n\ra S^m$ of $(n,p,k)$-type
we have in Theorem \ref{npk} that
\begin{equation}\label{sing2}
\la=\sqrt{\f{k(k+n-1)}{p}}
\end{equation}
is the nonzero singular value of $f_*$ at each point, and
from Theorem \ref{g2} that
$f=i\circ \pi$ where $\pi:(S^n,g_n)\ra (P,h)$ is a harmonic Riemannian submersion and $i:(P,h)\ra (S^m,\la^{-2}g_m)$ is an isometric minimal immersion.

Let $x\in S^n$,
$\la_1=\cdots=\la_p=\la$ and $\la_{p+1}=\cdots=\la_n=0$.
Then under an S-basis $\{\ep_1,\cdots,\ep_n\}$ of $(T_x S^n,g_n)$
for $f$
subject to $\la_1,\cdots,\la_n$,
\begin{equation}\label{hr}
h_r(\ep_i,\ep_j)=(r^2+\rho^2\la_i^2)\de_{ij}=\left\{\begin{array}{ll}
0 & i\neq j,\\
r^2+\rho^2\la^2 & 1\leq i=j\leq p,\\
r^2 & p+1\leq i=j\leq n.
\end{array}\right.
\end{equation}
Set $\mu:=r^2+\rho^2\la^2$. Then
$\pi:(S^n,h_r)\ra (P,\mu h)$ is a Riemannian submersion and $i:(P,\mu h)\ra (S^m,\mu\la^{-2}g_m)$ is an isometric minimal immersion.
Further, by Lemma \ref{lem3}.
we gain the harmonicity of $\pi:(S^n,h_r)\ra (P,\mu h)$
as in the proof of Theorem \ref{npk}.
Employing the composition formula (\ref{com2}),
we can show that $f=i\circ \pi$ is a harmonic map
from $(S^n,h_r)$ into $(S^m,g_m)$ for each $r\in U$.
Hence we have a simpler version of Theorem \ref{min_graph1} for LOMSEs.

\begin{thm}\label{min_graph2}
For an LOMSE $f:S^n\ra S^m$ and smooth $\rho:U\subset (0,+\infty)\ra \R$,
$M_{f,\rho}$ is minimal in $\R^{n+m+2}$ if and only if
%$\rho$ satisfies the ODE:
\begin{equation}\label{ODE1}
\f{\rho_{rr}}{1+\rho_r^2}+\f{(n-p)\rho_r}{r}+\f{p(\f{\rho_r}{r}-\f{\la^2\rho}{r^2})}{1+\f{\la^2\rho^2}{r^2}}=0.
\end{equation}

\end{thm}

\textbf{Remark. } For $f=H^{2m-1,m}$ with $m=2,4$ or $8$,
$M_{f,\rho}$ is minimal if and only if
\begin{equation}\label{ODE4}
\f{\rho_{rr}}{1+\rho_r^2}+\f{(m-1)\rho_r}{r}+\f{m(\f{\rho_r}{r}-\f{4\rho}{r^2})}{1+\f{4\rho^2}{r^2}}=0.
\end{equation}
This ODE was first obtained by Ding-Yuan \cite{d-y} based on the symmetry of Hopf maps.
It should be pointed out that the argument in the present paper is also applicable to non-equivariant $f$.

%Let
%\begin{equation}E_i:=\left\{\begin{array}{ll}
%(r^2+\rho(r)^2\la^2)^{-1/2}\ep_i & 1\leq i\leq p,\\
%r^{-1}\ep_i & p+1\leq i\leq n.
%\end{array}\right.
%\end{equation}
%Then $\{E_1,\cdots,E_n\}$ is an orthonormal basis of $(T_x S^n,h_r)$.

Let us analyze (\ref{ODE1}).
As in \cite{d-y}, set
\begin{equation}\label{varphi}
\varphi:=\f{\rho}{r},\quad t:=\log r.
\end{equation}
With
\begin{equation}\label{rhor}
\rho_r=(e^t\varphi)_t \f{dt}{dr}=\varphi_t+\varphi
\end{equation}
and
\begin{equation}
\rho_{rr}=(\varphi_t+\varphi)_t \f{dt}{dr}=\f{\varphi_{tt}+\varphi_t}{r},
\end{equation}
we can rewrite \eqref{ODE1} as
\begin{equation}
\f{\varphi_{tt}+\varphi_t}{1+(\varphi_t+\varphi)^2}+(n-p)(\varphi_t+\varphi)+\f{p(\varphi_t+\varphi-\la^2\varphi)}{1+\la^2 \varphi^2}=0.
\end{equation}
After introducing
\begin{equation}\label{psi}
\psi:=\varphi_t,
\end{equation}
we transform \eqref{ODE1} to the ODE system:
\begin{equation}\label{ODE2}
\aligned
\left\{\begin{array}{ll}
\varphi_t=\psi,\\
\psi_t=-\psi-\Big[\big(n-p+\f{p}{1+\la^2\varphi^2}\big)\psi+\big(n-p+\f{(1-\la^2)p}{1+\la^2\varphi^2}\big)\varphi\Big]\big[1+(\varphi+\psi)^2\big].
\end{array}\right.
\endaligned
\end{equation}
This is an autonomous system
and
$\g: t\mapsto (\varphi(t),\psi(t))$ satisfies \eqref{ODE2} if and only
$\g$ is an integral curve of the vector field
%\begin{equation}\label{ODE3}
$X:=(X_1,X_2)$
%\end{equation}
where
\begin{equation}\label{X12}\aligned
\left\{\begin{array}{ll}
X_1&=\psi,\\
X_2&=-\psi-\Big[\big(n-p+\f{p}{1+\la^2\varphi^2}\big)\psi+\big(n-p+\f{(1-\la^2)p}{1+\la^2\varphi^2}\big)\varphi\Big]\big[1+(\varphi+\psi)^2\big].
\end{array}\right.
\endaligned
\end{equation}

Clearly, $X$ has exactly 3 zero points $(0,0)$ and $(\pm\varphi_0,0)$, where
\begin{equation}\label{varphi0}
\varphi_0:=\sqrt{\f{p-n\la^{-2}}{n-p}}.
\end{equation}
Since $X$ is symmetric about the origin (i.e. $X(-\varphi,\psi)=-X(\varphi,\psi)$), we shall therefore only focus on the half plane $\varphi\geq0$.

\textbf{Remark. }As a zero point, $t\in \R\mapsto (\varphi_0,0)$ gives a trivial solution to (\ref{ODE2}). Hence
$\rho(r)=\varphi_0 r$ is a solution to (\ref{ODE1}) and
 $F_{f,\rho}:\R^{n+1}\ra \R^{m+1}$
$$F_{f,\rho}(y)=\left\{\begin{array}{cc}
\varphi_0|y|f(\f{y}{|y|}) & y\neq 0\\
0 & y=0
\end{array}\right.$$
is a Lipschitz solution to the minimal surface equations. Comparing (\ref{th}) and \eqref{varphi0}
we see $\varphi_0=\tan \th$
and the corresponding graph is exactly the $C_{f,\th}$.
Meanwhile, another trivial solution $t\in \R\mapsto (0,0)$ of \eqref{ODE2} gives the coordinate $(n+1)$-plane.

At $(0,0)$, the linearized system of (\ref{ODE2}) is
\begin{equation}
\left(\begin{array}{c}\varphi_t\\ \psi_t\end{array}\right) =A\left(\begin{array}{c}\varphi\\ \psi\end{array}\right)
\end{equation}
where
\begin{equation}\label{A}
A=\left(\begin{array}{cc}
0 & 1\\
\la^2p-n & -n-1
\end{array}
\right)=
\left(\begin{array}{cc}
0 & 1\\
k(k+n-1)-n & -n-1
\end{array}
\right).
\end{equation}
%Denote by $\mu_1,\mu_2$ the eigenvalues of $A$, then
%$$\mu_1+\mu_2=\text{tr }A=-n-1$$ and
%$$\mu_1\mu_2=|A|=n-\la^2p=n-k(k+n-1)<0.$$
%( It follows from Corollary \ref{npk} that $\la^2 p=k(k+n-1)$ with $k\geq 2$.)
Through calculations, the eigenvalues of $A$  are
\begin{equation}\label{la12}
\mu_1=k-1,\qquad \mu_2=-n-k,
\end{equation}
with eigenvectors
\begin{equation}\label{V12}
V_1:=(1, \mu_1)^T,\quad V_2:=(1, \mu_2)^T,
\end{equation}
respectively.
Hence $(0,0)$ is a saddle critical point.

At $(\varphi_0,0)$, the linearized system is
\begin{equation}
\left(\begin{array}{c}(\varphi-\varphi_0)_t\\ \psi_t\end{array}\right)=B\left(\begin{array}{c}\varphi-\varphi_0\\ \psi\end{array}\right)
\end{equation}
with
\begin{equation}
B=\left(\begin{array}{cc}
0 & 1\\
a & b
\end{array}
\right),
\end{equation}
where
\begin{equation}\label{ab}\aligned
a&:=\f{2\la^2(1-\la^2)p\varphi_0^2(1+\varphi_0^2)}{(1+\la^2\varphi_0^2)^2}=2n\Big(\f{n}{k(k+n-1)}-1\Big),\\
b&:=-1-(n-p+\f{p}{1+\la^2\varphi_0^2})(1+\varphi_0^2)=-n-1.
\endaligned
\end{equation}
Let $\mu_3,\mu_4$ be the eigenvalues of $B$.
Then $\mu_3+\mu_4=\text{tr }B=b<0$, $\mu_3\mu_4=|B|=-a>0$, and
\begin{equation}%\aligned
(\mu_3-\mu_4)^2=(\mu_3+\mu_4)^2-4\mu_3\mu_4=b^2+4a
=n^2-6n+1+\f{8n^2}{k(k+n-1)}.
%\endaligned
\end{equation}
When $n=3$, $k\geq 4$ or $n=5,k\geq 6$,
$\{\mu_3,\mu_4\}$ become a pair of conjugate complex numbers with negative real part;
while in other cases, both $\mu_3$ and $\mu_4$ are negative real numbers. Therefore
\begin{enumerate}

\item[(I)] If $(n,p,k)=(3,2,2),(5,4,2),(5,4,4)$ or $n\geq 7$, $(\varphi_0,0)$ is a stable center of (\ref{ODE2});
% as $t\ra +\infty$;

\item[(II)] If $(n,p)=(3,2)$, $k\geq 4$ or $(n,p)=(5,4),k\geq 6$, $(\varphi_0,0)$ is a stable spiral point of (\ref{ODE2}).% as $t\ra +\infty$.

\end{enumerate}

Based on the above local analysis,
we are able to establish
the following two existence results of nontrivial bounded solutions of
distinct types to \eqref{ODE2} according to the values of $(n,p,k)$.
The idea is to construct suitable barrier functions.
Since their proofs are a bit long and subtle, we leave them in Appendices \S \ref{App3}-\ref{App4}.

\begin{pro}\label{case1}
If $(n,p,k)=(3,2,2),(5,4,2),(5,4,4)$ or $n\geq 7$, then there exists a smooth solution $t\in \R\mapsto (\varphi(t),\psi(t))$ to (\ref{ODE2}), with properties
\begin{itemize}
\item $\lim\limits_{t\ra -\infty}(\varphi(t),\psi(t))=(0,0)$;
\item $\varphi(t)=O(e^{(k-1) t})$ and $\psi(t)=O(e^{(k-1) t})$ as $t\ra -\infty$;
\item $\lim\limits_{t\ra +\infty}(\varphi(t),\psi(t))=(\varphi_0,0)$;
\item $t\mapsto \varphi(t)$ is a strictly increasing function;
\item $\psi(t)> 0$ for every $t\in \R$.
\end{itemize}
%Namely, the orbit of the solution in the $\varphi\psi$-plane is a smooth curve above the $\varphi$-axis with endpoints $(0,0)$ and $(\varphi_0,0)$.
$$\includegraphics[scale=0.65]{F1.png}$$

\end{pro}

\begin{pro}\label{case2}
If $(n,p)=(3,2)$, $k\geq 4$ or $(n,p)=(5,4)$, $k\geq 6$, then there exist a smooth solution $t\in \R \mapsto (\varphi(t),\psi(t))$ to (\ref{ODE2})
and a strictly increasing sequence $\{T_i:i\in \Bbb{Z}^+\}$ in $\R$, such that
\begin{itemize}
\item $\lim\limits_{t\ra -\infty}(\varphi(t),\psi(t))=(0,0)$;
\item $\varphi(t)=O(e^{(k-1) t})$ and $\psi(t)=O(e^{(k-1) t})$ as $t\ra -\infty$;
\item $\lim\limits_{t\ra +\infty}(\varphi(t),\psi(t))=(\varphi_0,0)$;
\item $\lim\limits_{i\ra \infty}T_i=+\infty$;
\item $\psi(T_i)=0$ for all $i\in \Bbb{Z}^+$;
\item With $\varphi_i:=\varphi(T_i)$, $\{\varphi_{2m-1}:m\in \Bbb{Z}^+\}$
is  strictly decreasing
and
$\{\varphi_{2m}:m\in \Bbb{Z}^+\}$
is strictly increasing
with the common limit $\varphi_0$;
\item $\psi(t)>0$ for $t\in (-\infty,T_1)\cup \left(\bigcup\limits_{m\in \Bbb{Z}^+}(T_{2m},T_{2m+1})\right)$;
%, $t\in (-\infty,T_1)\mapsto \varphi(t)$
%and $t\in (T_{2m},T_{2m+1})\mapsto \varphi(t)$
%($m\in \Bbb{Z}^+$) are all strictly increasing functions;
\item $\psi(t)<0$ for $t\in \bigcup\limits_{m\in \Bbb{Z}^+}(T_{2m-1},T_{2m})$;
% $t\in (T_{2m-1},T_{2m})\mapsto \varphi(t)$ is strictly decreasing for each $m\in \Bbb{Z}^+$;
\item $(\varphi+\psi)(t)>0$ for all $t\in \R$.
\end{itemize}
Namely, the orbit of this solution tends to the saddle point $(0,0)$ as $t\ra -\infty$ and spins around the spiral point $(\varphi_0,0)$
% for infinitely
%many times
as $t\ra +\infty$.
$$\includegraphics[scale=0.65]{F2N.png}$$
\end{pro}

From Propositions \ref{case1}-\ref{case2},
we obtain the existence of entire minimal graphs associated to LOMSEs of $(n,p,k)$-type as follows.

\begin{thm}\label{graph}
For every LOMSE $f$, there exists a smooth function $\rho$ on $(0,\infty)$
such that
\begin{equation}\label{expF}
F_{f,\rho}(y)=\left\{\begin{array}{cc}
\rho(|y|)f(\f{y}{|y|}) & y\neq 0,\\
0 & y=0.
\end{array}\right.
\end{equation}
gives an entire minimal graph with $C_{f,\th}$, the LOC associated to $f$, as its tangent cone at infinity.
\end{thm}

\begin{proof}
Let $t\in \R\mapsto (\varphi(t),\psi(t))$ be the solution to (\ref{ODE2}) in Proposition \ref{case1}
or \ref{case2}.
Then
\begin{equation}
\rho(r)=r\cdot \varphi(\log r),\text{\ \ for \ } r\in (0,+\infty)\mapsto \rho(r)
\end{equation}
satisfies \eqref{ODE1}.
By Theorem \ref{min_graph2}, $M_{f,\rho}$
is a minimal submanifold in $\R^{n+m+2}$.
Moreover, as $r\ra 0$,
\begin{equation}
\rho(r)=r\cdot \varphi(\log r)=O(r^{k}),
\end{equation}
\begin{equation}
\rho_r=\varphi(\log r)+\psi(\log r)=O(r^{k-1}).
\end{equation}
Hence $F_{f,\rho}$ is $C^1$ at the origin
and,
thus by Theorem 6.8.1 in \cite{mo},
real analytic through the origin.

In addition, by Propositions \ref{case1}-\ref{case2}, $\varphi(t)\ra \varphi_0=\tan{\th}$ as $t\ra +\infty$.
Therefore, the LOC $C_{f,\th}$ is the unique tangent cone of the graph of $F_{f,\rho}$ at infinity.
\end{proof}

\bigskip

\subsection{Non-uniqueness and non-minimizing of minimal graphs}
The amusing spiral asymptotic behavior of the solutions in
Proposition \ref{case2}
produce the following interesting corollaries.
They explain the non-uniqueness of analytic solutions to the corresponding Dirichlet problem
and the non-minimizing property of those LOCs.

\begin{cor}\label{uni}
For an LOMSE $f$ of $(n,p,k)$-type with $(n,p)=(3,2)$, $k\geq 4$ or $(n,p)=(5,4)$, $k\geq 6$,
there exist infinitely many analytic solutions to the Dirichlet problem for boundary data $f_{\varphi_0}:=\varphi_0\cdot f$.
\end{cor}

\begin{proof}
For the solution $t\in \R\mapsto (\varphi(t),\psi(t))$ to \eqref{ODE2} in Proposition \ref{case2},
define $\{t_i\}$ to be the increasing sequence
with
$\varphi(t_i)=\varphi_0$.
Set $d_i=e^{t_i}$ and recall
\begin{equation}\label{Fd_i}
F_{f,\rho_{d_i}}(y)=\f{1}{d_i}F_{f,\rho}(d_i\cdot y),\quad \text{ for } y\in \mathbb D^{n+1} \text{ and } i\in\Bbb{Z}^+.
\end{equation}
Since the minimality is rescaling invariant,
$\{F_{f,\rho_{d_i}}:i\in \mathbb Z^+\}$ give
infinitely many analytic solutions to the minimal surface equations,
with (see \eqref{expF})
\begin{equation}
F_{f,\rho_{d_i}}(x)=\frac{\rho(d_i)}{d_i}f(x)=\varphi(t_i)f(x)=\varphi_0 \cdot f(x),\quad \text{ for } x\in \p \mathbb D^{n+1}.
\end{equation}
Hence we accomplish the proof.
\end{proof}

\textbf{Remark.}
 Similarly, {\it for each $\varphi\in [0,\varphi_1]$, there exists at least one analytic solution to
the Dirichlet problem for $f_\varphi:=\varphi\cdot f$; and moreover, for $\varphi\in [\varphi_2,\varphi_1)$
such solutions are not unique.}

\bigskip

\begin{cor}\label{non-min}
For an LOMSE $f$ of $(n,p,k)$-type with $(n,p)=(3,2)$, $k\geq 4$ or $(n,p)=(5,4)$, $k\geq 6$,
the LOC $C_{f,\th}$ is non-minimizing.
\end{cor}

\begin{proof}
Let $M$ be the graph of $F_{f,\rho}$.
Then the density function of $M$ 'centered at the origin'
is
$\Th: \R^+\ra \R$ by
\begin{equation}\label{density1}
\Th(R)=\f{\text{Vol}\big(M\cap \Bbb{D}^{n+m+2}(R)\big)}{\om_{n+1}R^{n+1}},
\end{equation}
where $\om_{n+1}$ denotes the volume of the unit ball in $\R^{n+1}$.

Denote by $M_i$ the graph of $F_{f,\rho_{d_i}}$ in \eqref{Fd_i}
and $\Th_i:=
\Th\Big(\sqrt{d_i^2+\rho(d_i)^2}\Big)$.
Then
\begin{equation}
\Th_i=
\f{d_i^{n+1}\text{Vol}(M_i)}{\om_{n+1}\big(\sqrt{d_i^2+\rho(d_i)^2}\big)^{n+1}}
=\f{\text{Vol}\left(M_i\right)}{\om_{n+1}\big(\sqrt{1+\tan^2\th}\big)^{n+1}}.
\end{equation}
By the monotonicity theorem for minimal submnaifolds (see e.g. \cite{c-m,fe}),
these quantities increasingly approach the density
$\Th_0$
of $C_{f,\rho}$ $-$ the tangent cone of $M$ at infinity,
i.e.,
%As for a cone, its density in particular equals
\begin{equation}
\Th_1\leq \cdots\leq \Th_k\cdots \ra \ \Th_0=\f{\text{Vol}\Big(C_{f,\th}\cap \Bbb{D}^{n+m+2}\big(\sqrt{1+\tan^2\th}\big)\Big)}{\om_{n+1}\big(\sqrt{1+\tan^2\th}\big)^{n+1}}.
\end{equation}
If
$\Th_1=\cdots=\Th_0$, then
$M$ must be a cone, which is not the case.
So $\Th_1<\Th_0$
and
$$
\text{Vol}\big(M_1\big)
<
\text{Vol}\Big(C_{f,\th}\cap \Bbb{D}^{n+m+2}\big(\sqrt{1+\tan^2\th}\big)\Big).
$$
Since
$$
\p\big(M_1\big)
=
\p \Big(C_{f,\th}\cap \Bbb{D}^{n+m+2}\big(\sqrt{1+\tan^2\th}\big)\Big),
$$
it follows consequently that
$C_{f,\th}$ is not area-minimizing.
\end{proof}

\bigskip\bigskip

\Section{Appendix}{Appendix}

\subsection{Proof of Lemma \ref{lem2}}\label{App1}
For $\phi:(\bar{N},\bar g)\rightarrow (N,g)$ and $x\in \text{Im}(\phi)\subset N$, the fiber $\phi^{-1}(x)$ over $x$ is a compact 
submanifold of $\bar{N}$ with finitely many connected components. This follows from the constant rank theorem (see e.g. \S II.7 of \cite{bo})
and the compactness of $\bar{N}$.
Let $P$ be the set of connected components of all fibers of $\phi$. More precisely, for $\bar{x}\in \phi^{-1}(x)$, denote by
$[\bar{x}]$ the connected component
of $\phi^{-1}(x)$ containing $\bar{x}$, then
\begin{equation}
P=\{[\bar{x}]:\bar{x}\in \bar{N}\}.
\end{equation}
Define
\begin{equation}
\pi(\bar{x})=[\bar{x}]\ \text{\ \ and\ \ \ } i([\bar{x}])=x.
\end{equation}
Then each fiber of $\pi$ is connected, and $\phi=i\circ \pi$.

Let $\bar{d}$ and $d$ be the intrinsic distance functions on $(\bar{N},\bar{g})$ and $(N,g)$, respectively, and $d_H$ be the \textit{Hausdorff distance function} (see e.g. \S 9.1 of \cite{fa}) on $P$, i.e.
\begin{equation}\label{dh1}
d_H([\bar{x}_0],[\bar{y}_0])=\max\{\sup_{\bar{x}\in [\bar{x}_0]}\inf_{\bar{y}\in [\bar{y}_0]}\bar{d}(\bar{x},\bar{y}),
\sup_{\bar{y}\in [\bar{y}_0]}\inf_{\bar{x}\in [\bar{x}_0]}\bar{d}(\bar{y},\bar{x})\}<+\infty.
\end{equation}
Then $(P,d_H)$ is a metric space equipped with the induced metric topology.

Given $[\bar{x}_0],[\bar{y}_0]\in P$, where the representatives $\bar{x}_0$ and $\bar{y}_0$ are chosen so that
$$\bar{d}(\bar{x}_0,\bar{y}_0)=\bar{d}([\bar{x}_0],[\bar{y}_0]):=\inf\{\bar{d}(\bar{x},\bar{y}):\bar{x}\in [\bar{x}_0],\bar{y}\in [\bar{y}_0]\},$$
let $\bar{\xi}:[0,1]\ra \bar{N}$ be a shortest
geodesic from $\bar{x}_0$ to $\bar{y}_0$ and $\xi:=\phi\circ \bar{\xi}$.
Due to the assumption on singular values, $(\phi_*)_{\bar{x}}: \left((\ker(\phi_*)_{\bar{x}})^\bot,\bar{g}\right)\subset (T_{\bar{x}} \bar{N},\bar{g})\ra (T_{x}N,g)$
is an isometric embedding for each $\bar{x}\in \bar{N}$. Then for each $\bar{x}\in [\bar{x}_0]$, there exists
a unique
 smooth curve $\bar{\xi}_{\bar{x}}:[0,1]\ra \bar{N}$, such that $\phi\circ \bar{\xi}_{\bar{x}}=\xi$, $\bar{\xi}_{\bar{x}}(0)=\bar{x}$ and $\bar{\xi}_{\bar{x}}'(t)$ is orthogonal to the fiber of $\phi$ going through $\bar{\xi}_{\bar{x}}(t)$. Denote
 \begin{equation}
 \Phi(\bar{x})=\bar{\xi}_{\bar{x}}(1).
 \end{equation}
 Noting that $\bar{\xi}_{\bar{x}}$ smoothly dependents on $\bar{x}$ and $\text{Length}(\bar{\xi}_{\bar{x}})=\text{Length}
 (\xi)=\text{Length}(\bar{\xi})$, we conclude that:
\begin{enumerate}
\item[(A)] $\Phi$ is a diffeomorphism between $[\bar{x}_0]$ and $[\bar{y}_0]$;
\item[(B)] $\bar{d}(\bar{x},\Phi(\bar{x}))=\bar{d}([\bar{x}_0],[\bar{y}_0])=d_H([\bar{x}_0],[\bar{y}_0])$ for each $\bar{x}\in [\bar{x}_0]$.
\end{enumerate}

Due to the compactness of $\bar{N}$, applying the constant rank theorem implies the existence of a positive constant $\de$, such that:
\begin{itemize}
\item[($\star$)] For each $\bar{x}\in \bar{N}$ with $x:=\phi(\bar{x})$, $\phi(\td{B}_\de(\bar{x}))$ is a $p$-dimensional embedded submanifold of $N$, and $\td{B}_\de(\bar{x})\cap \phi^{-1}(x)\subset [\bar{x}]$, where $\td{B}_r(\bar{x})$ is the geodesic ball centered at $\bar{x}$ and of radius $r$.
\end{itemize}

Denote by $B_r([\bar{x}])\subset P$ the metric ball centered at $[\bar{x}]$ and of radius $r$. Based on (A)-(B), we can derive the following
results through a contradiction argument:
\begin{enumerate}
\item [(C)] $i|_{B_{\de/2}([\bar{x}])}$ is injective;
\item [(D)] $i(B_{\de/2}([\bar{x}]))=\phi(\td{B}_{\de/2}(\bar{x}))$ is a $p$-dimensional embedded submanifold of $N$.
\end{enumerate}
 Therefore, we can easily endow $P$ with
a differential structure, so that both $i$ and $\pi$ are smooth maps. Moreover, letting $h:=i^*g$ implies that $\pi$ is a Riemannian submersion from $(\bar{N},\bar{g})$
 onto $(P,h)$ and $i$ is an isometric immersion from $(P,h)$ into $(N,g)$. It is worth noting that $d_H$
 is just the intrinsic distance function on $(P,h)$. This completes the proof of Lemma \ref{lem2}.

 \bigskip

\subsection{Proof of Corollary \ref{cor2}}\label{App2}

Suppose $f:S^n\ra S^m$ is an LOM with singular values $\la_1,\cdots,\la_n$ at $x\in S^n$.
Let
$\{\ep_1,\cdots,\ep_n\}$ and $\{e_1,\cdots,e_n\}$ be corresponding S-bases of $(T_x S^n,g_n)$ and $(T_x S^n,g)$, respectively.
Set
\begin{equation}\label{basis4}
E_j=(I_{f,\th})_* e_j=\f{(\cos\th \ep_j,\sin\th f_* \ep_j)}{\sqrt{\cos^2\th+\sin^2\th \la_j^2}},\qquad \forall 1\leq j\leq n.
\end{equation}
Then $\{E_1,\cdots,E_n\}$ forms an orthonormal basis of $T_{I_{f,\th}(x)} M_{f,\th}$, and
$$*(\nu_1\w \cdots\w \nu_{m+1})=\mathbf{X}\w E_1\w \cdots\w E_n$$
where $*$ is the {\it Hodge star} operator and $\{\nu_1,\cdots,\nu_{m+1}\}$ is an oriented orthonormal basis of the normal plane $N_{I_{f,\th}(x)} M_{f,\th}$.

Let $\{\ep_{n+2},\cdots,\ep_{n+m+2}\}$ be an oriented orthonormal basis of $Q_0:=\{x_1=\cdots=x_{n+1}=0\}$
and
$\a$ the angle between $Q_0$ and $N_{I_{f,\th}(x)} M_{f,\th}$.
Then
\begin{equation}
\aligned
\cos\a&=\lan \nu_1\w \cdots\w \nu_{m+1},\ep_{n+2}\w\cdots\w \ep_{n+m+2}\ran\\
&=\lan *(\nu_1\w \cdots\w \nu_{m+1}), *(\ep_{n+2}\w\cdots\w \ep_{n+m+2})\ran\\
&=\lan \mathbf{X}\w E_1\w \cdots \w E_n,\mathbf{Y}_1\w \ep_1\w \cdots\w \ep_n\ran\\
&=\left|\begin{array}{cccc}
\lan \mathbf{X},\mathbf Y_1\ran & \lan \mathbf{X},\ep_1\ran & \cdots & \lan\mathbf{X},\ep_n\ran\\
\lan E_1,\mathbf Y_1\ran & \lan E_1,\ep_1\ran & \cdots & \lan E_1,\ep_n\ran\\
& \cdots & &\\
\lan E_n,\mathbf Y_1\ran & \lan E_n,\ep_1\ran & \cdots & \lan E_n,\ep_n\ran
\end{array}\right|\\
&=\cos\th \prod_{j=1}^n \f{\cos\th}{\sqrt{\cos^2\th +\sin^2\th \la_j^2}}
\endaligned
\end{equation}
By applying Theorem \ref{npk}, we obtain \eqref{angle}.

Note that on $(S^n,g)$
the volume form
\begin{equation}\label{dV}
dV=\sqrt{\det\big(g(\ep_j,\ep_k)\big)}\ \ep^*_1\w\cdots\w \ep^*_n=\prod_{j=1}^n \sqrt{\cos^2\th+\sin^2\th \la_j^2}\ \ep^*_1\w \cdots\w \ep^*_n.
\end{equation}
By integration over $S^n$, the fomula \eqref{volume} follows.

It is easy to see that, at $y=tI_{f,\theta}(x)$ for $t>0$,
$X, E_1,\cdots, E_n$ are precisely the angle directions (see \cite{wo} for definition) of $T_yC_{f,\theta}$ relative to $Q_0^\perp$,
with Jordan angles
\begin{equation}
\th_0=\th\quad \text{and}\quad \th_i=\arccos\left(\f{\cos\th}{\cos^2\th+\sin^2\th\la_i^2}\right)\quad \forall 1\leq i\leq n.
\end{equation}
As in \cite{x-y}\cite{j-x-y3}, the slope function of $C_{f,\th}$ is thereby
\begin{equation}
W=\prod_{j=0}^n \sec\th_j=\sec\th \prod_{j=1}^n \f{\sqrt{\cos^2\th +\sin^2\th \la_j^2}}{\cos\th}.
\end{equation}

 \bigskip

\subsection{Proof of Proposition \ref{case1}}\label{App3}

Let $D$ be the bounded closed domain on the $\varphi\psi$-plane enclosed by the line segment from $(0,0)$ to $(\varphi_0,0)$
and the graph of function $h:[0,\varphi_0]\ra \R$ given by
\begin{equation}
h(\varphi)=\f{\big(\f{(\la^2-1)p}{1+\la^2\varphi^2}-(n-p)\big)\varphi}{c(n-p)},
\end{equation}
where $c\in (0,1]$ is a constant to be chosen.
$$\includegraphics[scale=0.8]{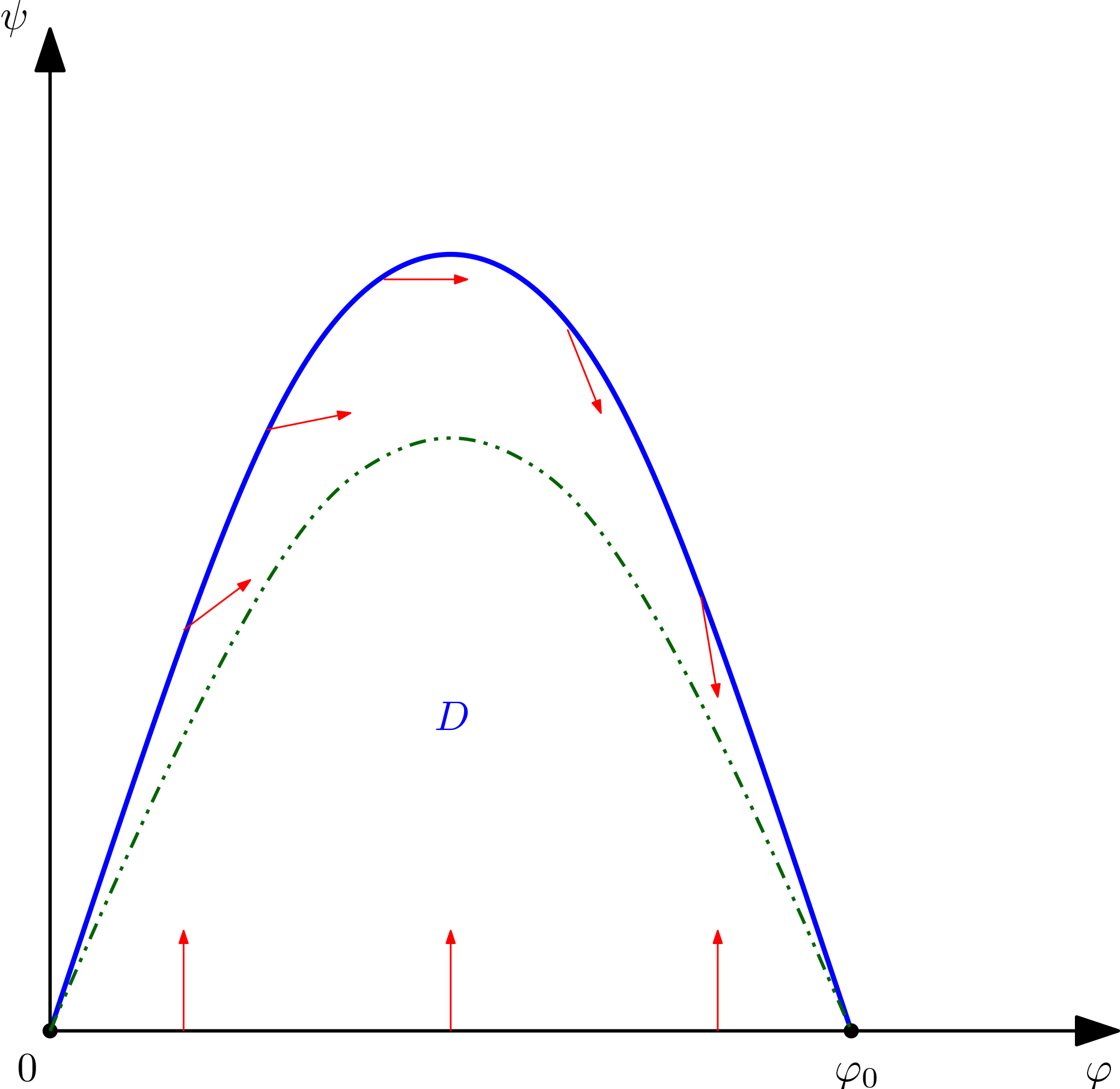}$$

 We shall prove that
$D$ is invariant under the forward development of (\ref{ODE2}) by verifying that $X=(X_1,X_2)$ points inward in $\p D$
except at the zero points $(0,0)$ and $(\varphi_0,0)$. In other words, we need to show:
\begin{enumerate}
%\item $X_1(0,\psi)>0$ for each $\psi\in (0,\f{n+1}{2}\varphi_0]$;
\item[(A)] $X_2(\varphi,0)>0$ for $\varphi\in (0,\varphi_0)$;
\item[(B)] $h'(\varphi)> \f{X_2}{X_1}(\varphi,h(\varphi))$ for $\varphi\in (0,\varphi_0)$.
\end{enumerate}

Here (A) is obvious and (B) requires following careful calculations.

%Whenever $\varphi\in (0,\varphi_0)$, $n-p+\f{(1-\la^2)p}{1+\la^2\varphi^2}<0$, hence
%$$X_2(\varphi,0)=-(n-p+\f{(1-\la^2)p}{1+\la^2\varphi^2})\varphi(1+\varphi^2)>0$$
%and (1) is proved.

Set
\begin{equation}\label{f12}
\aligned
f_1(\varphi)&:=\f{(\la^2-1)p}{1+\la^2\varphi^2}-(n-p),\\
f_2(\varphi)&:=n-p+\f{p}{1+\la^2\varphi^2},
\endaligned
\end{equation}
and
\begin{equation}\label{h}
h(\varphi)=\f{ f_1(\varphi)\varphi}{c(n-p)}.
\end{equation}
Then
\begin{equation}\label{dh}
h'(\varphi)=\f{f_1(\varphi)+ f'_1(\varphi)\varphi}{c(n-p)}.
\end{equation}
Due to (\ref{X12}), (\ref{f12}), (\ref{h}) and (\ref{dh}), (B) is equivalent to
\begin{equation}\label{neq}\aligned
0&< h'(\varphi)+1+ \Big[n-p+\f{p}{1+\la^2\varphi^2}+\big(n-p+\f{(1-\la^2)p}{1+\la^2\varphi^2}\big)\f{\varphi}{h(\varphi)}\Big]\big[1+(\varphi+h(\varphi))^2\big]\\
&=\left(1+\f{f_1(\varphi)}{c(n-p)}\right)+\f{ f'_1(\varphi)\varphi}{c(n-p)}+(f_2(\varphi)-c(n-p))\left[1+\varphi^2\big(1+\f{f_1(\varphi)}{c(n-p)}\big)^2\right]\\
:&=I+II+III\cdot IV.
\endaligned
\end{equation}

By (\ref{varphi0}),
\begin{equation}
\la^2\varphi_0^2=\f{\la^2 p-n}{n-p},\qquad 1+\la^2\varphi_0^2=\f{(\la^2-1)p}{n-p}.
\end{equation}
Set
\begin{equation}
s:=\f{1+\la^2\varphi_0^2}{1+\la^2\varphi^2}-1.
\end{equation}
Then $\varphi\in (0,\varphi_0)$ implies $s\in (0,\la^2\varphi_0^2)=(0,\f{\la^2 p-n}{n-p})$, and
\begin{equation}\aligned
\f{1}{1+\la^2\varphi^2}&=\f{1+s}{1+\la^2\varphi_0^2}=\f{n-p}{(\la^2-1)p}(1+s),\\
\la^2\varphi^2&=\f{1+\la^2\varphi_0^2}{1+s}-1=\f{\la^2\varphi_0^2-s}{1+s}=\f{\f{\la^2 p-n}{n-p}-s}{1+s}.
\endaligned
\end{equation}
It immediately follows that
\begin{equation}
f_1(\varphi)=(n-p)s,
\end{equation}
\begin{equation}
f_2(\varphi)=\f{(n-p)(\la^2+s)}{\la^2-1}
\end{equation}
and
\begin{equation}\aligned
 f'_1(\varphi)\varphi&=-\f{2(\la^2-1)p\la^2\varphi^2}{(1+\la^2\varphi^2)^2}\\
&=-2(\la^2-1)p\big(1-\f{1}{1+\la^2\varphi^2}\big)\f{1}{1+\la^2\varphi^2}\\
&=-\f{2(n-p)}{(\la^2-1)p}(\la^2 p-n-(n-p)s)(1+s).
\endaligned
\end{equation}
Therefore
\begin{equation}
I=1+\f{s}{c}:=I(s),
\end{equation}
\begin{equation}
II=-\f{2(n-p)}{c(\la^2-1)p}\left(\f{\la^2 p-n}{n-p}-s\right)(1+s):=II(s),
\end{equation}
\begin{equation}
III=\f{n-p}{\la^2-1}(\la^2-c(\la^2-1)+s):=III(s)
\end{equation}
and
\begin{equation}\aligned
IV&=1+\varphi^2(1+\f{f_1(\varphi)}{c(n-p)})^2\\
&=1+\la^{-2}(\la^2\varphi^2)(1+\f{f_1(\varphi)}{c(n-p)})(1+\f{f_1(\varphi)}{c(n-p)})\\
&=1+\la^{-2}(\f{\la^2 p-n}{n-p}-s)\f{1+\f{s}{c}}{1+s}(1+\f{s}{c})\\
&\geq 1+\la^{-2}(\f{\la^2 p-n}{n-p}-s)(1+\f{s}{c})\\
:&=IV(s).
\endaligned
\end{equation}
Let
\begin{equation}
F(s):=I(s)+II(s)+III(s)\cdot IV(s).
\end{equation}
By $c\in(0,1]$ and $s>0$, $III(s)>0$.
So $I+II+III\cdot IV\geq F(s)$ for $s\in (0,\la^2\varphi_0^2)$, i.e., $\varphi\in (0,\varphi_0)$.
Observe that $F(s)$
is a cubic polynomial in $s$ and the coefficient of the third order term is $-\f{n-p}{c\la^2(\la^2-1)}<0$. Hence $F(s)=F(0)+sG(s)$, where $G(s)$ is a quadratic polynomial whose graph is a parabola opening downward.
 This implies $G(s)\geq \min\{G(0),G(\la^2\varphi_0^2)\}$ for $s\in (0,\la^2\varphi_0^2)$. Therefore, for (\ref{neq}), it suffices to show
\begin{itemize}
\item $F(0)\geq 0$;
\item $G(0)> 0$;
\item $G(\la^2\varphi_0^2)> 0$.
\end{itemize}

A straightforward calculation shows
\begin{equation}
\aligned
F(0)&=I(0)+II(0)+III(0)\cdot IV(0)\\
&=1-\f{2(\la^2 p-n)}{c(\la^2-1)p}+(n-p)(\f{\la^2}{\la^2-1}-c)(1+\f{\la^2p-n}{\la^2(n-p)})\\
&=1+n-\f{2(\la^2p-n)}{c(\la^2-1)p}-\f{c(\la^2-1)n}{\la^2},
\endaligned
\end{equation}
\begin{equation}\label{G0}
\aligned
G(0)=&F'(0)=I'(0)+II'(0)+III'(0)\cdot IV(0)+III(0)\cdot IV'(0)\\
=&\f{1}{c}-\f{2(n-p)}{c(\la^2-1)p}(\f{\la^2 p-n}{n-p}-1)+\f{n-p}{\la^2-1}(1+\f{\la^2 p-n}{\la^2 (n-p)})\\
&+(n-p)(\f{\la^2}{\la^2-1}-c)\la^{-2}(\f{\la^2 p-n}{c(n-p)}-1)\\
%=&\f{1}{c}-\f{2(\la^2 p-n)}{c(\la^2-1)p}+\f{2(n-p)}{c(\la^2-1)p}+\f{n}{\la^2}+\f{\la^2 p-n}{c\la^2}(\f{\la^2}{\la^2-1}-c)-(n-p)(\f{\la^2}{\la^2-1}-c)\la^{-2}\\
%=&\f{1}{c}-\f{2}{c}(1-\f{n-p}{(\la^2-1)p})+\f{2(n-p)}{c(\la^2-1)p}+\f{n}{\la^2}+\f{p}{c}-\f{n-p}{c(\la^2-1)}-p+\f{n}{\la^2}-\f{n-p}{\la^2-1}+\f{c(n-p)}{\la^2}\\
=&(\f{1}{c}-1)p-\f{1}{c}+\f{(\f{4-p}{c}-p)(n-p)}{\la^2p -p}+\f{(c+2)n-cp}{\la^2},\\
\endaligned
\end{equation}
\begin{equation}
\aligned
F(\la^2\varphi_0^2)&=I(\la^2\varphi_0^2)+II(\la^2\varphi_0^2)+III(\la^2\varphi_0^2)\cdot IV(\la^2\varphi_0^2)\\
&=1+\f{\la^2 p-n}{c(n-p)}+\f{n-p}{\la^2-1}(\la^2-c(\la^2-1)+\f{\la^2p-n}{n-p})\\
&=1+n+\f{\la^2 p-n}{c(n-p)}-c(n-p)
\endaligned
\end{equation}
and
\begin{equation}
\aligned
G(\la^2\varphi_0^2)&=\f{F(\la^2\varphi_0^2)-F(0)}{\la^2\varphi_0^2}\\
&=\f{1}{c}+\f{c(n-p)}{\la^2}+\f{2(n-p)}{c(\la^2-1)p}>0.
\endaligned
\end{equation}

Recalling $\la^2=\frac{k(k+n-1)}{p}$, we choose $c$ according to the values of $(n,p,k)$:

\textbf{Case 1.} $(n,p,k)=(3,2,2)$.

Using $c=1$, we have
$$F(0)=4-\f{5}{3c}-\f{9c}{4}=\f{1}{12}>0$$
and
$$G(0)=\f{4}{3c}-\f{5}{6}+\f{c}{4}=\f{3}{4}>0.$$

\textbf{Case 2.} $(n,p,k)=(5,4,2)$.

With $c=1$,
$$F(0)=6-\f{7}{4c}-\f{10c}{3}=\f{11}{12}>0$$
and
$$G(0)=\f{3}{c}-\f{7}{6}+\f{c}{3}=\f{13}{6}>0.$$

\textbf{Case 3.} $(n,p,k)=(5,4,4)$.

For $c=\f{6}{7}$,
$$F(0)=6-\f{27}{14c}-\f{35c}{8}=0$$
and
$$G(0)=\f{3}{c}-\f{81}{28}+\f{c}{8}=\f{5}{7}>0.$$

\textbf{Case 4.} $n\geq 7$.

In this case, Theorem \ref{npk2} asserts $p<n<2p$ and $p\geq 4$.

Take $c=\f{1}{2}$. By $n>p$,
$$\aligned
F(0)&=1+n-\f{2(\la^2p-n)}{c(\la^2-1)p}-\f{c(\la^2-1)n}{\la^2}\\
&\geq 1+n-\f{2}{c}-cn=\f{n}{2}-3>0.
\endaligned$$
From $2p>n$ we have
\begin{equation*}\aligned
&(\la^2 p-p)-3(n-p)=k(k+n-1)-p-3(n-p)\\
\geq & 2(n+1)-p-3(n-p)=2p-n+2>0,
\endaligned
\end{equation*}
i.e., $\f{n-p}{\la^2 p-p}<\f{1}{3}$. Hence
\begin{equation*}\aligned
G(0)&=p-2+\f{(8-3p)(n-p)}{\la^2 p-p}+\f{5n-p}{2\la^2}\\
&>p-2+\f{8-3p}{3}>0
\endaligned
\end{equation*}

Therefore we establish (B) that $D$ is invariant under the forward development of (\ref{ODE2}).
Since $(0,0)$ is a saddle critical point, there exists a smooth solution $t\in (-\infty,T_\infty)\mapsto (\varphi(t),\psi(t))\in \R^2$ to (\ref{ODE2}), with
$\lim\limits_{t\ra -\infty}(\varphi(t),\psi(t))=(0,0)$. Here
$T_\infty\in \R\cup \{+\infty\}$ such that $(-\infty,T_\infty)$ is the maximal existence interval of this solution. Moreover, by Theorem 3.5 in \S VIII of \cite{h},
as $t\ra -\infty$,
$\varphi(t)=O(e^{\mu_1 t})$, $\psi(t)=O(e^{\mu_1 t})$ and the direction of $(\varphi(t),\psi(t))^T$
converges to that of $V_1$, i.e., an eigenvector of $A$ associated to $\mu_1$ (see (\ref{A}), (\ref{la12}) and (\ref{V12})).
 It is easy to check that $h'(0)>\mu_1$.
 Thus the orbit of this solution remains in $D$ and $T_\infty=+\infty$.
 By (A), we know $\varphi'(t)=\psi(t)>0$.
 Hence the $\om$-limit set of the orbit must be a critical point, not a limit cycle, as $t$ tends to positive infinity.
Now we complete the proof.
\bigskip

\subsection{Proof of Proposition \ref{case2}}\label{App4}

The proof relies heavily on the following lemma.

\begin{lem}\label{lem}
For $(n,p)=(3,2)$, $k\geq 4$ or $(n,p)=(5,4)$, $k\geq 6$,
let $t\in [b_0,b_2]\mapsto (\varphi(t),\psi(t))$ be a smooth solution to (\ref{ODE2}) and  $b_1\in (b_0,b_2)$ so that
\begin{itemize}
\item $\varphi(b_0)\geq \sqrt{\f{3p-n-1}{3(n-p)}}$;
\item $\psi(b_0)=\psi(b_1)=\psi(b_2)=0$;
\item $\psi(t)>0$ for $t\in (b_0,b_1)$, and $\psi(t)<0$ for $t\in (b_1,b_2)$.
\end{itemize}
Then $\varphi(b_1)>\varphi_0$ and $\varphi(b_0)<\varphi(b_2)<\varphi_0$.
\end{lem}

\textbf{Remark. }By this lemma, there are no limit cycles of (\ref{ODE2}) on the region $\varphi\geq\sqrt{\f{3p-n-1}{3(n-p)}}$.

\begin{proof}
Using symbols in Appendix \ref{App3},
we have from (\ref{ODE2}) that
\begin{equation}\label{psit}
\psi_t=-\psi-(f_2(\varphi)\psi-f_1(\varphi)\varphi)\big[1+(\varphi+\psi)^2\big].
\end{equation}
By assumptions, $\varphi(b_1)\neq \varphi_0$ and
$0\geq \psi'(b_1)=f_1(\varphi(b_1))\varphi(b_1)\big(1+\varphi(b_1)^2\big)$.
So $\varphi(b_1)>\varphi_0$ and $\psi'(b_1)<0$.
Similarly $\varphi(b_0),\varphi(b_2)<\varphi_0$.

For $t\in(b_1,b_2)$, with
\begin{equation}\label{td}
\left\{\begin{array}{ll}
\td{\varphi}=\varphi\\
\td{\psi}=-\psi
\end{array}\right.
\end{equation}
(\ref{ODE2}) becomes
\begin{equation}\label{forY}
\left\{\begin{array}{ll}
\td{\varphi}_t=-\td{\psi}\\
\td{\psi}_t=-\td{\psi}-\big(f_2(\td{\varphi})\td{\psi}+f_1(\td{\varphi})\td{\varphi}\big)\big[1+(\td{\varphi}-\td{\psi})^2\big]
\end{array}\right.
\end{equation}
By the monotonicity, $\psi$ for $t\in(b_0,b_1)$ and $\td{\psi}$ for $t\in(b_1,b_2)$
can be written as smooth functions $\psi(\varphi)$ and $\td{\psi}(\varphi)$ respectively.
Then
we have
\begin{equation}
\dfrac{d\psi}{d\varphi}=-1-\left[f_2(\varphi)-f_1(\varphi)\dfrac{\varphi}{\psi}\right]\left[1+\left(\varphi+\psi\right)^2\right],
\end{equation}
and
\begin{equation}
\dfrac{d\td{\psi}}{d\varphi}=\,\,1+\left[f_2(\varphi)+f_1(\varphi)\dfrac{\varphi}{\td{\psi}}\right]\left[1+\left(\varphi-\td{\psi}\right)^2\right].
\end{equation}
Therefore
\begin{equation}
\dfrac{d\psi}{d\varphi}-\dfrac{d\td{\psi}}{d\varphi}<
f_1(\varphi)
\left\{
\dfrac{\varphi}{\psi}\left[1+\left(\varphi+\psi\right)^2\right]-
\dfrac{\varphi}{\td{\psi}}\left[1+\left(\varphi-\td{\psi}\right)^2\right]
\right\}.
\end{equation}
Note that $\psi-\td{\psi}$ is continuous on $[\varphi_0,\varphi(b_1)]$ with value zero at $\varphi(b_1)$.
Through a contradiction argument, we have $\psi>\td{\psi}$ on $[\varphi_0,\varphi(b_1))$.
$$\includegraphics[scale=0.8]{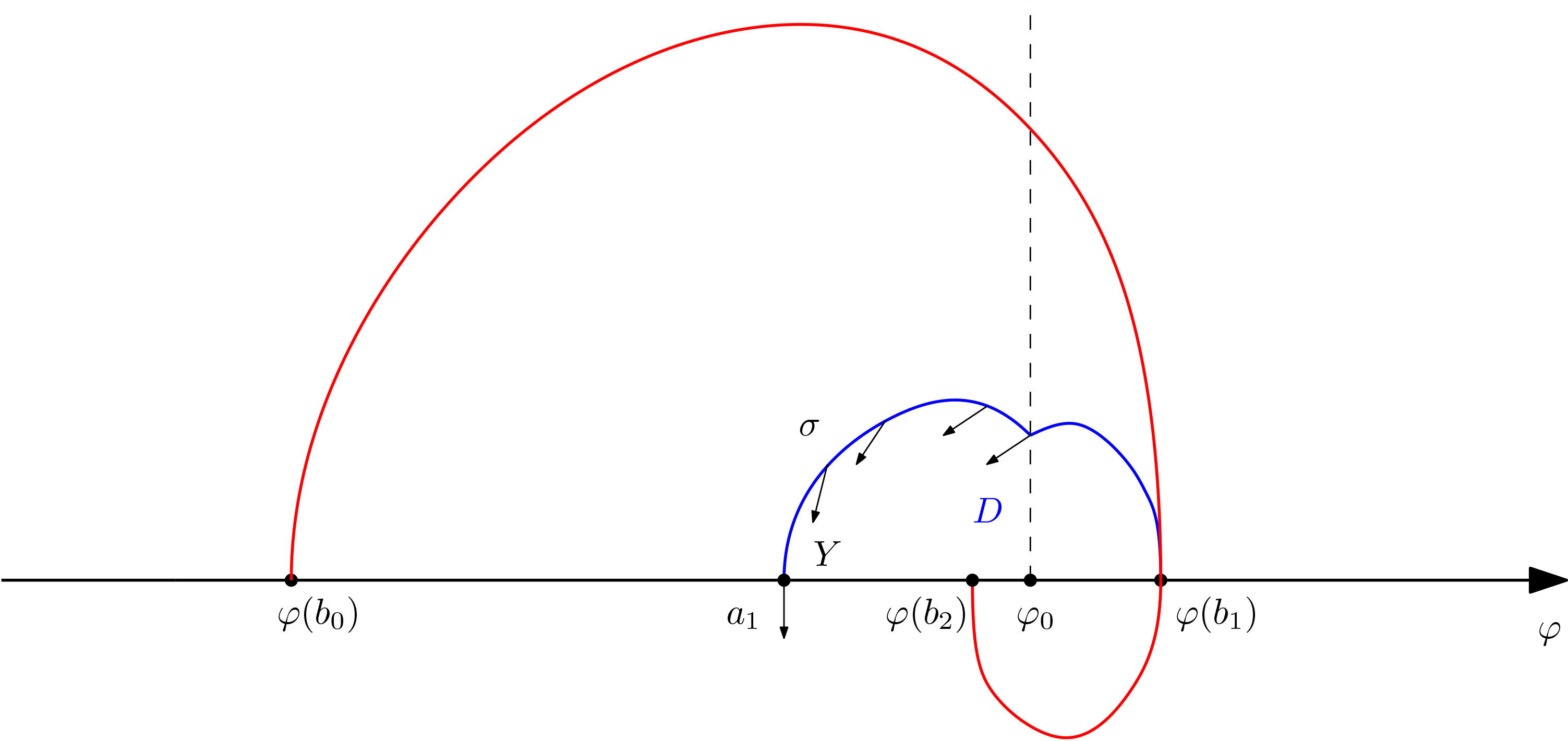}$$

Let $\sigma$ be the orbit of the backward solution to \eqref{ODE2} from $(\varphi_0, \td{\psi}(\varphi_0))$ to $(a_1,0)$ for some $a_1>\varphi(b_0)$.
Based on \eqref{forY}, set
\begin{equation}\label{Y12}
\aligned
Y_1(\varphi,\psi)&=-\psi,\\
Y_2(\varphi,\psi)&=-\psi-(f_2(\varphi)\psi+f_1(\varphi)\varphi)\big[1+(\varphi-\psi)^2\big].
\endaligned
\end{equation}
If we have
\begin{itemize}
\item [$(\star):$]
$\left|\begin{array}{cc} X_1 & X_2 \\ Y_1 & Y_2\end{array}\right|<0$ at $(\varphi,\psi)$ when $\varphi\geq \sqrt{\f{3p-n-1}{3(n-p)}}$
and $\psi>0$,
\end{itemize}
then the inequality holds for each point of $\sigma$.
Hence the region $D$ embraced by $\sigma$, the $\varphi$-axis and the striaight line $\varphi=\varphi_0$
forms an invariant set under the forward development of \eqref{forY}.
Therefore $\varphi(b_0)<a_1<\varphi(b_2)<\varphi_0$.

Since $X_1=-Y_1=\psi>0$,
$(\star)$ equals to saying that $Y_2+X_2<0$, which is obtained through careful calculations as follows.
$$\aligned
Y_2+X_2=&-\psi-(f_2(\varphi)\psi+f_1(\varphi)\varphi)\big[(1+(\varphi-\psi)^2\big]\\
&-\psi-(f_2(\varphi)\psi-f_1(\varphi)\varphi)\big[1+(\varphi+\psi)^2\big]\\
=&-2\psi-2f_2(\varphi)\psi(1+\varphi^2+\psi^2)+4f_1(\varphi)\varphi^2\psi\\
\leq & 2\psi(-1-f_2(\varphi)(1+\varphi^2)+2f_1(\varphi)\varphi^2)\\
=&\f{2\psi}{1+\la^2\varphi^2}\Big[-1-\la^2\varphi^2-(1+\varphi^2)\big((n-p)(1+\la^2\varphi^2)+p\big)\\
&\qquad\quad+2\varphi^2\big((\la^2-1)p-(n-p)(1+\la^2\varphi^2)\big)\Big]\\
=&\f{2\psi}{1+\la^2\varphi^2}\left[-3\la^2(n-p)\varphi^4+\big(\la^2(3p-n-1)-3n\big)\varphi^2-n-1\right]\\
<&-\f{6\la^2(n-p)\psi}{1+\la^2\varphi^2}\varphi^2(\varphi^2-\f{3p-n-1}{3(n-p)})\leq 0.
\endaligned$$
Now the proof of the lemma gets complete.
\end{proof}

%\renewcommand{\proofname}{\it Proof of Proposition \ref{case2}.}

%\begin{proof}

As in Appendix \ref{App3}, there exists a smooth solution $t\in(-\infty,T_\infty)\mapsto (\varphi(t),\psi(t))$ to (\ref{ODE2}), with $\lim\limits_{t\ra -\infty}(\varphi(t),\psi(t))=(0,0)$, $\varphi(t)=O(e^{\mu_1 t})$,
$\psi(t)=O(e^{\mu_1 t})$ and the direction of $(\varphi(t),\psi(t))^T$ convergent to that of $V_1$ as $t\ra -\infty$.
We shall accomplish the proof of Proposition
\ref{case2} in the following steps.

\textbf{Step 1. }\textit{Show the existence of $t_1\in (-\infty,T_\infty)\subset \R$, such that $\psi(t)>0$ for all $t\in (-\infty,t_1]$ with}
\begin{equation}\label{step1}
\varphi(t_1)=\varphi_0\quad \text{and }\quad \psi(t_1)\leq \f{1}{5}\varphi_0.
\end{equation}

Define $g:[0,\varphi_0]\ra \R$ by
\begin{equation}
g(\varphi)=(2f_1(\varphi)+\f{1}{5})\varphi.
\end{equation}
Let $D$ be the domain enclosed by
the graph of $g$, the $\varphi$-axis and the line $\varphi=\varphi_0$.

$$\includegraphics[scale=0.8]{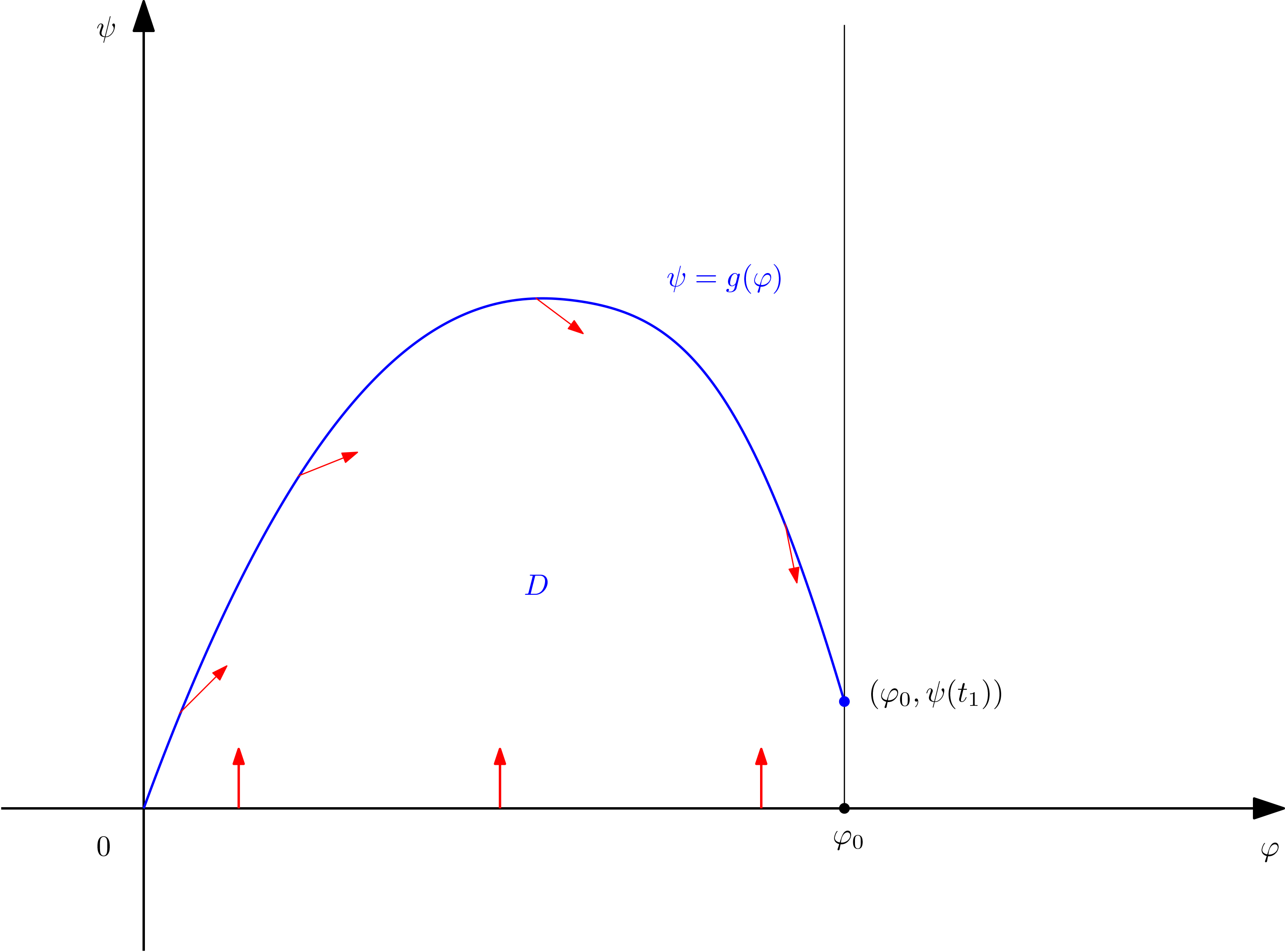}$$
We claim that  the vector field $X$ points inward on $\partial D-\{\varphi=\varphi_0\}$.
Namely,
\begin{enumerate}
\item[(A)] $X_2(\varphi,0)>0$ for each $\varphi\in (0,\varphi_0)$;
\item[(B)] $g'(\varphi)> \f{X_2}{X_1}(\varphi,g(\varphi))$ for any $\varphi\in (0,\varphi_0)$.
\end{enumerate}
Here (A) is trivial and (B) is equivalent to
\begin{equation}
\aligned
0<& g'(\varphi)+1+\left(f_2(\varphi)-\f{\varphi f_1(\varphi)}{g(\varphi)}\right)\big[1+(\varphi+g(\varphi))^2\big]\\
=&\left(\f{6}{5}+2f_1(\varphi)\right)-\Big(-2 f'_1(\varphi)\varphi\Big)+\left(f_2(\varphi)-\f{f_1(\varphi)}{2f_1(\varphi)+\f{1}{5}}\right)\left[1+\varphi^2(\f{6}{5}+2f_1(\varphi))^2\right]\\
:=&I-II+III\cdot IV.
\endaligned
\end{equation}
As in \S \ref{App3}, we use
$$s:=\f{1+\la^2\varphi_0^2}{1+\la^2\varphi^2}-1\quad \text{and} \quad s\in (0,\la^2\varphi_0^2).$$
Similarly, we have (now $n-p=1$ in our cases)
\begin{equation}
\aligned
I&=\f{6}{5}+2s,\\
II&=\f{4}{(\la^2-1)p}(\la^2p-n-s)(1+s),\\
III&=\f{\la^2+s}{\la^2-1}-\f{s}{2s+\f{1}{5}},\\
IV&=1+\f{\la^2p-n-s}{\la^2(1+s)}(\f{6}{5}+2s)^2=1+\f{(\la^2-1)p}{4\la^2}\left(\f{\f{6}{5}+2s}{1+s}\right)^2\cdot II.
\endaligned
\end{equation}
Therefore
\begin{equation}\label{es1}\aligned
&I-II+III\cdot IV\\
=&I+III+\left[\f{(\la^2-1)p}{4\la^2}\left(\f{\f{6}{5}+2s}{1+s}\right)^2\cdot III-1\right]\cdot II\\
\geq &I+III+\left[\f{p}{4}\left(\f{\f{6}{5}+2s}{1+s}\right)^2\left(1-\f{s}{2s+\f{1}{5}}\right)-1\right]\cdot II
\endaligned
\end{equation}
Set
\begin{equation}
F(s):=\left(\f{\f{6}{5}+2s}{1+s}\right)^2\left(1-\f{s}{2s+\f{1}{5}}\right)=\f{4}{25}\left(\f{3+5s}{1+s}\right)^2\f{1+5s}{1+10s}.
\end{equation}
Then
$$\log F=\log \f{4}{25}+2\log(3+5s)-2\log(1+s)+\log(1+5s)-\log(1+10s)$$
and
$$\aligned
\f{d\log F}{ds}&=\f{10}{3+5s}-\f{2}{1+s}+\f{5}{1+5s}-\f{10}{1+10s}\\
&=\f{-11+20s+175s^2}{(3+5s)(1+s)(1+5s)(1+10s)}.
\endaligned$$
 Hence $F'(s)=0(>0,<0)$ if and only if $s=\f{1}{5}(>\f{1}{5},<\f{1}{5})$, and
\begin{equation}\label{F}
\min_{s\in (0,\infty)}F=F(\f{1}{5})=\f{32}{27}.
\end{equation}

For $(n,p)=(5,4)$, substituting \eqref{F} into \eqref{es1} leads to
\begin{equation}
I-II+III\cdot IV\geq I+III+\f{5}{27}II> 0.
\end{equation}

For $(n,p)=(3,2)$, it then produces
\begin{equation}
\aligned
&I-II+III\cdot IV\geq I+III-\f{11}{27}II\\
=&\f{6}{5}+2s+\f{\la^2+s}{\la^2-1}-\f{s}{2s+\f{1}{5}}-\f{22}{27(\la^2-1)}(2\la^2-3-s)(1+s)\\
\geq& \f{6}{5}+2s+1-\f{s}{2s+\f{1}{5}}-\f{44}{27}(1+s)\\
=&\f{19}{270}+\f{10}{27}s+\f{1}{20s+2}> 0.
\endaligned
\end{equation}

Hence (B) holds for both cases.
%%%%%

Since $g'(0)>\mu_1$, the solution develops in $D$ until it hits
the border line $\varphi=\varphi_0$ at $t_1\in \R$
or it approaches $(\varphi_0,0)$ as $t\rightarrow +\infty$.
Due to the fact that $(\varphi_0,0)$ is a spiral point,
the latter cannot occur and moreover $t_1<+\infty$, $\varphi(t_1)=\varphi_0$ and $\psi(t_1)>0$.

\textbf{Step 2.} %\textit{Show the existence of $T_1\in \R$, such that $\psi(T_1)=0$, $\varphi_1:=\varphi(T_1)>\varphi_0$ and $\psi(t)>0$ for $t\in (-\infty,T_1)$.}
Before $\psi(t)$ reaches zero, we have $\varphi_t=\psi>0$, $\varphi>\varphi_0$ (after $t_1$)
and $f_1(\varphi)> 0$.
Consequently,
\begin{equation}
(\varphi+\psi)_t=-(f_2(\varphi)\psi-f_1(\varphi)\varphi)\big[1+(\varphi+\psi)^2\big]\leq 0.
\end{equation}
Hence the solution intersects the $\varphi$-axis for the first time when $t$ equals some $T_1\in \R$,
with $\varphi_0<\varphi_1:=\varphi(T_1)\leq \varphi_0+\psi(t_1)\leq \frac{6}{5}\varphi_0$.

\textbf{Step 3.}
At $t=T_1$, $\psi_t<0$.
So the solution dipps into the lower half plane and similarly cannot limits to $(\varphi_0,0)$.
By the argument in the proof of Lemma \ref{lem}, the solution extends forward to the $\varphi$-axis again (after $T_1$) when $t$ equals some $T_2\in \R$.
Mark $t_2\in(T_1,T_2)$ for $\varphi(t_2)=\varphi_0$.
 When $\psi<0$ and $\varphi\leq\varphi_0$, we have $(\varphi+\psi)_t\geq 0$.
 Therefore, $\varphi_0>\varphi_2:=\varphi(T_2)>\varphi_0+\psi(t_2)\geq \varphi_0-\psi(t_1)\geq \frac{4}{5}\varphi_0$.

\textbf{Step 4.}
By induction,
we obtain
$\{T_i:i\in \Bbb{Z}^+\}$
and $\varphi_i:=\varphi(T_i)$
with properties:
\begin{itemize}
%\item $\lim\limits_{i\ra \infty}T_i=+\infty$;
\item $\psi(T_i)=0$ for each $i\in \Bbb{Z}^+$;
\item $\{\varphi_{2m-1}:m\in \Bbb{Z}^+\}$ is a strictly decreasing sequence in $(\varphi_0,\frac{6}{5}\varphi_0]$,\\ $\{\varphi_{2m}:m\in \Bbb{Z}^+\}$
is a strictly increasing sequence in $[\f{4}{5}\varphi_0,\varphi_0)$;
\item $\psi(t)>0$ in $(-\infty,T_1)\cup\left(\bigcup\limits_{m\in \Bbb{Z}^+}(T_{2m},T_{2m+1})\right)$;
\item $\psi(t)<0$ in $\bigcup\limits_{m\in \Bbb{Z}^+}(T_{2m-1},T_{2m})$.
\end{itemize}

\textbf{Step 5.}
Assume $a:=\lim_{m\ra \infty}\varphi(T_{2m})<\varphi_0$.
Then there would be a limit cycle for \eqref{ODE2} through $(a,0)$.
But $a>\f{4}{5}\varphi_0>\sqrt{\f{3p-n-1}{3(n-p)}}$.
It leads to a contradiction to the nonexistence of limit cycles in Lemma \ref{lem}.
Therefore, $\lim\limits_{m\ra \infty}\varphi(T_{2m})=\lim\limits_{m\ra \infty}\varphi(T_{2m-1})=\varphi_0$.

Since the solution cannot attain $(\varphi_0,0)$ in a finite time,
it is now clear that $T_i\rightarrow +\infty$.
This
completes the proof of Proposition \ref{case2}.

%\end{proof}

\bibliographystyle{amsplain}

\end{document}